\documentclass[11pt]{amsart} 

\usepackage[english]{babel}
\usepackage[utf8]{inputenc}
\usepackage[T1]{fontenc}
\usepackage{hyperref}
\usepackage{amsmath}
\usepackage{amssymb}
\usepackage{amsfonts}
\usepackage{amssymb}
\usepackage{amsthm}
\usepackage{amscd}
\usepackage{graphicx}
\usepackage{xcolor}
\usepackage{float}
\usepackage{slashed}

\numberwithin{equation}{section}

\definecolor{darkblue2}{RGB}{0, 0, 200}
\definecolor{darkblue}{RGB}{0, 0, 100}
\definecolor{darkestblue}{RGB}{0, 0, 50}
\colorlet{lightblue}{darkblue!50!}
\definecolor{darkred}{RGB}{100, 0, 0}
\colorlet{lightred}{darkred!50!}
\definecolor{darkgreen}{RGB}{0,100,0}

\usepackage{geometry}

\usepackage{comment}

\newtheorem{definition}{Definition}[section]
\newtheorem{lemma}[definition]{Lemma}
\newtheorem{theorem}[definition]{Theorem}
\newtheorem{proposition}[definition]{Proposition}
\newtheorem{corollary}[definition]{Corollary}
\newtheorem{example}[definition]{Example}
\newtheorem{remark}[definition]{Remark}

\DeclareMathOperator{\Ric}{Ric}

\begin{document}

\title[Spectral Torical Band Inequalities and Black Hole Existence]
{Spectral Torical Band Inequalities and Generalizations of
the Schoen-Yau Black Hole Existence Theorem}

\author[Hirsch]{Sven Hirsch}
\address{Department of Mathematics, Duke University, Durham, NC, 27708, USA}
\email{sven.hirsch@duke.edu}

\author[Kazaras]{Demetre Kazaras}
\address{Department of Mathematics, Duke University, Durham, NC, 27708, USA}
\email{demetre.kazaras@duke.edu}

\author[Khuri]{Marcus Khuri}
\address{Department of Mathematics, Stony Brook University, Stony Brook, NY, 11794, USA}
\email{khuri@math.sunysb.edu}

\author[Zhang]{Yiyue Zhang}
\address{Department of Mathematics, University of California, Irvine, CA, 92697, USA}
\email{yiyuez4@uci.edu}

\thanks{M. Khuri acknowledges the support of NSF Grant DMS-2104229.}

\begin{abstract}
Generalized torical band inequalities give precise upper bounds for the width of compact manifolds with boundary in terms of positive pointwise lower bounds for scalar curvature, assuming certain topological conditions. We extend several incarnations of these results in which pointwise scalar curvature bounds are replaced with spectral scalar curvature bounds. More precisely, we prove upper bounds for the width in terms of the principal eigenvalue of the operator $-\Delta +cR$, where $R$ denotes scalar curvature and $c>0$ is a constant. Three separate strategies are employed to obtain distinct results holding in different dimensions and under varying hypotheses, namely we utilize spacetime harmonic functions, $\mu$-bubbles, and spinorial Callias operators. In dimension 3, where the strongest result is produced, we are also able to treat open and incomplete manifolds, and establish the appropriate rigidity statements. Additionally, a version of such spectral torus band inequalities is given where tori are replaced with cubes. Finally, as a corollary
we generalize classical work of Schoen and Yau, on the existence of black holes due to concentration of matter, to higher dimensions and with alternate measurements of size.
\end{abstract}

\maketitle

\section{Introduction} \label{sec:intro}

In \cite{GromovLawson1}, Gromov-Lawson introduced a homotopy theoretic obstruction to positive scalar curvature on closed spin manifolds, referred to as \textit{enlargeability}. Informally, this notion contends that since the Ricci endomorphism must have at least one positive eigenvalue at each point when scalar curvature is positive, the manifold cannot expand dramatically in all directions simultaneously. The $n$-torus $T^n$, therefore, cannot admit a metric of positive scalar curvature because it may be viewed as expanding in all directions by passing to covers. This heuristic is exemplified in the so called \textit{torus band inequality}. More precisely, if the product $T^{n-1}\times [-1,1]$ admits a Riemannian metric having scalar curvature bounded below by $\lambda>0$, then the manifold's width or rather distance between the two boundary components is bounded above by 
\begin{equation}\label{tbin}
\mathrm{width}\leq 2\pi \sqrt{\frac{n-1}{n\lambda}}. 
\end{equation}
This sharp inequality was first proved by Gromov in \cite{Gromov1} for $n\leq 7$ using minimal hypersurface techniques, and was extended to all dimensions by Cecchini \cite{Cecchini1} and Zeidler \cite{Zeidler1} using spinorial methods involving Callias operators. A variety of related band-width inequalities were established by Cecchini-Zeidler \cite{CecchiniZeidler1} again using spinors, and by R\"{a}de \cite{Rade1} with the $\mu$-bubble approach. Furthermore, in \cite{HKKZ} spacetime harmonic functions are applied to obtain a version of the 3-dimensional torus band inequality
with rigidity statement, and Chai-Wan \cite{CW} have established results of this type in the setting of initial data sets for the Einstein equations.

In the current paper we present spectral versions of torical band inequalities, as well as Gromov's cube inequality \cite[Section 3.8]{Gromov2}, and show how these can be used to obtain generalizations of the Schoen-Yau \cite{SY} black hole existence result. In what follows, all manifolds are assumed to be connected, oriented, Hausdorff,
second-countable, and smooth. Given an $n$-dimensional Riemannian manifold $(M^n,g)$ and a number $c\in\mathbb{R}$, we define the \textit{$c$-spectral constant}
by
\begin{equation}\label{variation}
\Lambda_c=\inf\left\{\int_{M^n}\left(|\nabla u|^2+cRu^2\right)dV \text{ }\Big| \text{  $u\in H_0^1( M^n)$, $\int_{M^n} u^2dV=1$}\right\},
\end{equation}
where $R$ denotes scalar curvature and $H^1_0(M^n)$ is the Sobolev space of $L^2$ functions with square integrable derivatives arising as the completion of $C^{\infty}_0(M^n)$, the space of smooth functions with compact support, in the Sobolev $H^1$-norm. When $M^n$ is a compact manifold with boundary, $\Lambda_c$ is defined as the $c$-spectral constant of the interior $\mathring{M}^n$ which coincides with the principal Dirichlet eigenvalue of the Schr\"{o}dinger operator $-\Delta+cR$, and the condition $\Lambda_c >0$ may be interpreted as a weak notion of positive scalar curvature if $c>0$. This particular type of Schr\"{o}dinger operator appears in various geometric contexts for different values of $c$. The particular choice $c=\frac{1}{2}$ plays a special role in the search for black holes, while other values of $c$ are used for the Yamabe problem, minimal surfaces, and Ricci flow with surgery; we refer to the article by Li-Mantoulidis \cite{LM} for an extended discussion. 

The first spectral band-width result presented below is restricted to dimension 3, but provides the strongest statement and conclusions. In particular, we are able to treat open (possibly incomplete) manifolds and obtain rigidity in the case of equality, for an infinite range of $c$ values. This theorem is obtained using the level set technique involving spacetime harmonic functions. If $E$ is a non-empty collection of ends associated with a manifold $M^n$, and $\Sigma^{n-1}\subset M^n$ is a closed hypersurface, then the distance between $E$ and $\Sigma^{n-1}$ will be labelled by $d(E,\Sigma^{n-1})$ and is defined as the infimum of lengths of paths traveling from points in $\Sigma^{n-1}$ to $E$. For further details concerning the notion of ends and properties of open Riemannian manifolds, we refer to \cite[Appendix C]{HKKZ}.

\begin{theorem}\label{sh}
Let $(M^3, g)$ be an open 3-dimensional Riemannian manifold with a smooth closed hypersurface $\Sigma^2$ separating the ends of $M^3$ into two disjoint nonempty classes $E_-$ and $E_+$. 
Assume that there are no spherical classes in $H_2(M^3;\mathbb{Z})$, and that the scalar curvature of $(M^3,g)$ is bounded from below $\inf_{M^3} R>-\infty$. If $c>\frac16$ and $\Lambda_c(g)>0$ then
\begin{equation}
d(E_-,\Sigma^2)+d(E_+,\Sigma^2)\le \frac{\pi}{\alpha}, \quad\textnormal{\; where \;} \alpha=\sqrt{\frac{\Lambda_c(6-c^{-1})}{2c(8-c^{-1})}}. 
\label{rad 3}
\end{equation}
Moreover, equality is achieved in \eqref{rad 3} if and only if $(M^3,g)$ is isometric to the warped product
\begin{equation}
\left(\left(0,\frac{\pi}{\alpha}\right)\times \Sigma^2,\;  d\rho^2+ [\sin(\alpha \rho)]^\frac{8c-2}{6c-1} g_{0}\right),
\end{equation}
where $(\Sigma^2,g_0)$ is a flat torus.
\end{theorem}

It should be noted that the model geometries exhibit different asymptotic behavior at the ends depending on whether $\frac{1}{6}<c<\frac{1}{4}$, $c>\frac{1}{4}$, or $c=\frac{1}{4}$, namely the cross-sectional tori either expand, contract, or remain unchanged respectively, see Figure \ref{equality}. Moreover, if we assume the pointwise bound $R\geq\lambda>0$ and note that $\Lambda_c \geq c\lambda$, then applying Theorem \ref{sh} while letting $c\rightarrow\infty$ recovers the original torus band inequality \eqref{tbin}. 



In order to treat higher dimensional spectral band-width inequalities, we will employ the use of spinorial Callias operators \cite{CecchiniZeidler1}. These techniques, which involve modified Dirac equations, have similarities with Witten's proof of the spacetime version of the positive mass theorem \cite{Witten}.
The statement of the next result requires certain terminology. A compact Riemannian manifold $(M^n,g)$ whose boundary components are separated into two disjoint and non-empty collections $\partial M^n=\partial_- M^n\sqcup \partial_+ M^n$ will be referred to as a \textit{Riemannian band}, and its \textit{width} is defined to be the distance between the two classes of boundary components $d(\partial_- M^n,\partial_+ M^n)$. A Riemannian band is called \textit{overtorical} if there exists a smooth map $F:M^n\to T^{n-1}\times[-1,1]$ of nonzero degree, with 
$F(\partial_\pm M^n)\subset T^{n-1}\times \{\pm1\}$. Furthermore, a Riemannian band which is spin is said to be \textit{$\hat{A}$-overtorical} \cite[Section 5]{Zeidler} if there is an integer $k\geq 1$ and a smooth map $F:M^n\to T^{k-1}\times[-1,1]$ such that $F(\partial_\pm M^n)\subset T^{k-1}\times\{\pm1\}$, and the A-hat genus $\hat{A}(F^{-1}(p))\neq 0$ for regular values $p$ of $F$; this latter condition is equivalent to requiring that the $\hat{A}$-degree of $F$ not vanish. Notice that in order for the $\hat{A}$-genus of the fiber to be nonzero, the number $k$ must be less than or equal to $n$ and satisfy $n-k=0$ mod $4$. For instance, the product of a $\mathrm{K3}$ surface with an interval is an $\hat A$-overtorical band, where the map $F$ may be taken to be projection to the interval. If $k=n$, then the $\hat A$-degree agrees with the usual degree of a map between oriented manifolds, and in this situation an $\hat A$-overtorical band is an overtorical band. 

\begin{theorem}\label{T:spinors}
Let $(M^n,\partial_{\pm}M^n, g)$ be an odd dimensional $\hat{A}$-overtorical band with $n\geq 1$. If $c>\frac{n-1}{4n}$ and $\Lambda_c>0$ then
\begin{equation}\label{aoiwfoinah}
d(\partial_- M^n,\partial_+M^n)\le 2\pi\sqrt{\frac{c}{\Lambda_c}\left(\frac{(4c-1)n+2-4c}{(4c-1)n+1}\right)}.
\end{equation}
\end{theorem}

Recall that for $n\geq 3$ the conformal Laplacian is given by $-\Delta+c_n R$, where $c_n = \frac{n-2}{4(n-1)}$. Thus, the lower bound for $c$ given by $\frac{n-1}{4n}$ coincides with the conformal Laplacian constant of one dimension higher $c_{n+1}$. The pointwise version of this result was obtained by Zeidler in \cite[Theorem 3.1, Proposition 5.5]{Zeidler}, and states that if $R\geq \lambda>0$ then the $\hat{A}$-overtorical band width satisfies the upper bound of \eqref{tbin}. As with Theorem \ref{sh}, the pointwise analogue may be obtained from the spectral result by observing that $\Lambda_c \geq c\lambda$ and then sending $c\rightarrow\infty$. 

We may remove the spin assumption up to dimension 7 by utilizing (warped) $\mu$-bubbles. These hypersurfaces, introduced by Gromov \cite[Section 5]{Gromov2}, satisfy a type of prescribed mean curvature equation and come with a stability property that can be exploited in a similar manner to the classical Schoen-Yau usage of stable minimal surfaces. Alternatively, from a mathematical general relativity perspective, the $\mu$-bubbles may be viewed as a stable apparent horizon within an auxiliary initial data set for the Einstein equations. In the next theorem, we establish a spectral band width inequality restricted to the case $c=\frac{1}{2}$. The pointwise version of this result, that is under the assumption $R\geq\lambda>0$, again yields the same upper bound as in \eqref{tbin} and is given by Gromov \cite[page 8]{Gromov1} with a proof via torical symmetrization.
Moreover, the pointwise rendition may also be obtained from the work of Rad\"{e} \cite{Rade1} who also exploited $\mu$-bubbles to obtain a variety of band-width estimates, or separately by modifications of the arguments presented below in Section \ref{sec4}.

\begin{theorem}\label{T:mu bubble}
Let $(M^n,\partial_{\pm}M^n,g)$ be an overtorical band with $n\leq 7$. If $\Lambda_{\frac{1}{2}}>0$ then
\begin{equation}\label{aohfoiaopijhpoq}
d(\partial_- M^n,\partial_+ M^n)\le\pi\sqrt{\frac{2n}{(n+1)\Lambda_{\frac{1}{2}}}}.
\end{equation}
\end{theorem}

Another type of width inequality has been obtained for cubes by Gromov \cite[Section 3.8]{Gromov2} for dimensions $n\leq 8$, by minimal surface techniques, and this was extended to all higher dimensions by Wang-Xie-Yu \cite[Theorem 1.1]{WXY} (see also \cite{Xie}) with Dirac operator methods. The result states that if a Riemannian metric on the cube $[-1,1]^n$ has scalar curvature bounded below by $R\geq\lambda>0$, then 
\begin{equation}\label{aofoiqnoihnpqoin}
    \sum_{i=1}^n\frac1{\ell_i^2}\geq\frac{n\lambda}{4\pi^2(n-1)}
\end{equation}
where $\ell_i$ is the distance between the $i$th pair of opposite faces of the cube; the constant $\frac{1}{4\pi^2}$ is optimal \cite[Remark 2.1]{WXY}. The inverse square root of the quantity on the left-hand side of \eqref{aofoiqnoihnpqoin}
is referred to as the \textit{cubical-width}. 
Here we establish a spectral version of the cube-width inequality for the case when $c=\frac{1}{2}$.

\begin{theorem}\label{spectral cube}
Let $([-1,1]^n,g)$ be a Riemannian cube. If $\Lambda_{\frac{1}{2}}>0$ then
\begin{equation}
    \sum_{i=1}^n\frac1{\ell_i^2}\geq\frac{(n+1)\Lambda_{\frac{1}{2}}}{2\pi^2n},
\end{equation}
where $\ell_i$ is the distance between the $i$th pair of opposite faces of the cube.
\end{theorem}

This result implies the spectral torus-band inequality in all dimensions, namely if $(T^{n-1}\times [-1,1],g)$ satisfies $\Lambda_{\frac{1}{2}}>0$ then the width satisfies the upper bound \eqref{aohfoiaopijhpoq}. Indeed, the torus-band naturally gives rise to a Riemannan $n$-cube, and since the spectral constant of the cube is no less than that of the parent torus-band, it is positive. We may then apply Theorem \ref{spectral cube}, and utilize the fact that the torus-band width is less than or equal to the distance between the corresponding pair of opposite faces in the cube, to obtain the desired estimate.
Unlike the other results presented so far, the proof of Theorem \ref{spectral cube} consists of showing how the spectral inequality follows from the pointwise inequality by passing to a warped product constructed with the principal eigenfunction,
in similarity to part of the torical symmetrization process. This method of proof also extends to the spectral inequality for $T^{n-1}\times [-1,1]$, giving an alternative proof to that mentioned above. Although this approach is quite simple, it only applies to the case when $c=\frac{1}{2}$, and is not well-suited for rigidity statements such as in Theorem \ref{sh}.

While the spectral torical-band type inequalities are of independent interest, it is our intention to apply them here to obtain black hole existence results, particularly in higher dimensions. In \cite[Theorem 2]{SY} (see also \cite{Yau}), Schoen-Yau obtained such a result for 3-dimensional initial data sets, which depends on a particular notion of radius. Given a region $\Omega$, consider a simple closed curve $\Gamma\subset\Omega$ which bounds a disc. Let $\mathbf{r}$ denote the supremum of values $r$ with the property that the $r$-distance neighborhood from $\Gamma$ does not intersect $\partial\Omega$, and $\Gamma$ does not bound a disc in this neighborhood. The Schoen-Yau radius $\mathrm{Rad}_{sy}(\Omega)$ is then defined to be the supremum of $\mathbf{r}$ among all curves $\Gamma$ as above. 
An initial data set for the Einstein equations consists of a triple $(M^n,g,k)$, where $(M^n,g)$ is a Riemannian $n$-manifold and $k$ is a symmetric 2-tensor on $M^n$ representing the extrinsic curvature of the embedding into spacetime. By taking traces of the Gauss-Codazzi relations, these quantities satisfy the constraint equations
\begin{equation}\label{qoihr0pqijh}
2\mu=R+\left(\mathrm{Tr}_g k\right)^2 -|k|^2,\quad\quad\quad J=\operatorname{div}_g \left(k-(\mathrm{Tr}_g k)g\right),
\end{equation}
where $\mu$, $J$ represent the matter energy and momentum densities respectively. Suppose that $M^3$ is compact, with boundary satisfying the \textit{untrapped condition} $H>|\mathrm{Tr}_{\partial M^3} k|$ where $H$ denotes the (outward) boundary mean curvature, and $\mu-|J|\geq\Lambda>0$ on a bounded domain $\Omega \subset\subset \mathring{M}^3$. 
The Schoen-Yau black hole existence theorem states that if 
\begin{equation}\label{qoinoinwqh}
\mathrm{Rad}_{sy}(\Omega)\geq \pi\sqrt{\frac{3}{2\Lambda}},
\end{equation}
then $M^3$ contains an apparent horizon $\Sigma^2$. These surfaces, which are alternatively known as marginally outer or inner trapped surfaces, satisfy one of the equations $H_{\Sigma^2}\pm \mathrm{Tr}_{\Sigma^2}k=0$; we refer to \cite{Lee} for further properties of apparent horizons and their physical significance. Thus, for a region of fixed size measured by the radius, sufficient concentration of matter induces gravitational collapse. This yields a manifestation of Thorne's \textit{hoop conjecture} \cite{Thorne}. An advantageous feature of this result, also shared by Theorem \ref{bhexistence} below, is that it applies under quite general conditions. This separates it from most other results on this topic, which require special hypotheses such as symmetry or maximality ($\mathrm{Tr}_g k=0$) of the initial data, for instance \cite{BeigOMurchadha,BizonMalecOMurchadha1,Khuri,Khuri1,Malec1,Wald1}.

The proof in \cite{SY} proceeds in two steps. The first is to establish a spectral-radius inequality, and the second consists of employing this estimate to show that \eqref{qoinoinwqh} forces blow-up in the solution of Jang's equation on $M^3$. Here Jang's equation refers to the quasi-linear elliptic equation of prescribed mean curvature type, used heavily in their proof of the spacetime version of the positive mass theorem \cite{SY2}. We will follow a similar prescription, with alternate notions of radii motivated by the spectral torical-band width inequalities described above.  
A Riemannian band $(N^n ,\partial_{\pm}N^n ,h)$ will be referred to as a \textit{nonPSC-band} if $\partial_- N^n$ and $\partial_+ N^n$ are not separable by a smooth embedded hypersurface $\Sigma^{n-1}\subset N^n$ which admits a metric of positive scalar curvature. As is discussed at the end of Section \ref{sec4}, overtorical bands are examples of nonPSC-bands for $n\leq 8$.
The \textit{torical-radius} $\mathrm{Rad}_t(\Omega)$ is defined to be the supremum of widths of all nonPSC-bands $(N^n ,\partial_{\pm}N^n ,h)$ that are isometrically immersed into $\Omega$, and the \textit{cubical-radius} $\mathrm{Rad}_c(\Omega)$ is defined to be the supremum of cubical-widths of all cubes $([-1,1]^n,h)$ that are isometrically immersed into $\Omega$.

\begin{theorem}\label{bhexistence}
Let $3\leq n\leq7$, and suppose that $(M^n,g,k)$ is a compact $n$-dimensional initial data set with untrapped boundary. Assume that there is a constant $\Lambda>0$ and a compact submanifold $\Omega\subset \mathring{M}^n$ with Lipschitz boundary, such that $\mu-|J|\geq\Lambda$ on $\Omega$. If 
\begin{equation}\label{q3j0qhp0hjpoihjnaoi}
    \mathrm{Rad}(\Omega)\geq \pi \sqrt{\frac{2n}{(n+1)\Lambda}}
\end{equation}
where $\mathrm{Rad}$ is either the torical-radius $\mathrm{Rad}_t$ or the cubical-radius $\mathrm{Rad}_c$, 
then there exists a closed properly embedded smooth apparent horizon $\mathcal{S}^{n-1}$ within $M^n$. Moreover, if $\mu-|J|\geq \lambda>0$ on the apparent horizon then it is of positive Yamabe type with $\mathrm{Rad}(\mathcal{S}^{n-1})\leq \pi\sqrt{\frac{2(n-1)}{n\lambda}}$. In particular,
if the apparent horizon lies within $\Omega$ then its radius satisfies the estimate with $\lambda=\Lambda$. 
\end{theorem}

In Section \ref{sec6}, a class of initial data will be constructed which satisfy the hypotheses of this theorem.
It should be noted that, analogously to the Schoen-Yau result \cite{SY}, there are no examples to be found among those that are maximal. To see this, note that if the data were maximal, then the constraint equations \eqref{qoihr0pqijh} and the inequality $\mu-|J|\geq\Lambda$ imply that $R\geq 2\Lambda$ on $\Omega$. Thus, any overtorical band isometrically immersed in $\Omega$ must have width no greater than $\pi\sqrt{\frac{2(n-1)}{n\Lambda}}$ by the pointwise overtorical band-width inequality, but this precludes $\mathrm{Rad}(\Omega)$ from achieving \eqref{q3j0qhp0hjpoihjnaoi}. 
See Shi-Tam \cite{ShiTam} (and the realted \cite{ALY}) as well as \cite{KhuriXie} for black hole existence statements in the time symmetric case, when $k=0$. A comparison between the torical-radius and the Schoen-Yau radius will also be given in the last section. In particular, it is shown that $\mathrm{Rad}_t(\Omega)\geq\mathrm{Rad}_{sy}(\Omega)$ for any region $\Omega$, and therefore Theorem \ref{bhexistence} in dimension 3 recovers \cite[Theorem 2]{SY}. We would also like to point out contemporaneous work by Chow and Wan
\cite{ChowWan} that involves similar results.

We close the introduction with an immediate consequence of Theorem \ref{bhexistence}, which has the advantage that
each side of the black hole existence criteria is straightforward and, in principle, relatively easy to compute. 
Moreover, it utilizes cubes which are topologically balls, and thus emulates the essence of Thorne's hoop conjecture
\cite{Thorne} which posits that gravitational collapse occurs when enough mass is compressed inside a perfect sphere.

\begin{corollary}\label{cor main}
Let $3\leq n\leq7$, and suppose that $(M^n,g,k)$ is an asymptotically flat $n$-dimensional initial data set. Assume that there is an $n$-cube within $M^n$ on which
\begin{equation}
    \mu-|J|\geq \frac{2n\pi^2}{n+1} \sum_{i=1}^n\frac1{\ell_i^2},
\end{equation}
where $\ell_i$ is the distance between the $i$th pair of opposite faces of the cube. Then the data
contains a closed properly embedded smooth apparent horizon.
\end{corollary}

This paper is organized as follows. The proofs of the spectral torical band-width inequalities, namely Theorems \ref{sh}, \ref{T:spinors}, and \ref{T:mu bubble}, will be given in Sections \ref{sec2}, \ref{sec3}, and \ref{sec4} respectively.
The spectral cube inequality, Theorem \ref{spectral cube}, will be presented in Section \ref{sec5}. Moreover, as mentioned above, Section \ref{sec6} is dedicated to the black hole existence result Theorem \ref{bhexistence}. Additionally, an appendix is provided that addresses certain existence and regularity issues concerning warped $\mu$-bubbles.

\medskip
\noindent\textbf{Acknowledgements.} The authors would like to thank Hubert Bray, Simon Brendle, and Richard Schoen for insightful discussions, and their interest in this work.

\section{The Spacetime Harmonic Function Approach}
\label{sec2}

In this section we will utilize the technique of spacetime harmonic functions to establish the spectral torical
band-width inequality of Theorem \ref{sh}. Such functions arise as solutions to a semi-linear elliptic equation associated with initial data sets, and were introduced in \cite{HKK} within the context of the spacetime version of
the positive mass theorem. Applications to comparison geometry were recently studied in \cite{HKKZ}.

\subsection{Background}

In this paper we will only use a special case of spacetime harmonic functions in which the associated auxiliary initial data set is umbilic. Thus, the spacetime harmonic equation itself will take as input a single function $f$ defined on a band, which shall be chosen later to extract advantageous coercive behavior of the solution. The following proposition provides the basic existence result for this special class of spacetime harmonic functions, and is an immediate consequence of the more general existence result discussed in \cite[Section 4]{HKK}. 

\begin{proposition}\label{p:existence} Let $(M^n,\partial_\pm M^n,g)$ be an $n$-dimensional Riemannian band, and consider a function $f\in \mathrm{Lip}(M^n)$, as well as constants $c_-<c_+$. Then for any $\varsigma\in(0,1)$, there exists a unique solution $u\in C^{2,\varsigma}(M^n)$ of the spacetime harmonic Dirichlet problem
\begin{equation}\label{e:bandspacetimeharmoniceq1}
\begin{cases}
    \Delta u+nf|\nabla u|=0&\text{ in }M^n ,\\
    u= c_\pm&\text{ on }\partial_\pm M^n .
\end{cases}
\end{equation}
\end{proposition}

We note a basic technical fact concerning spacetime harmonic functions, which is shared by solutions to other elliptic equations, namely their set of critical points is small. This becomes useful when expressing certain integral inequalities below, which involve dividing by $|\nabla u|$.

\begin{proposition}\label{critical}
Let $u$ be a nontrivial spacetime harmonic function, with Lipschitz $f$, on a Riemannian manifold $(M^n,g)$, $n\geq 2$.
Then the critical set $\{x\in M^n\mid \nabla u(x)=0\}$ is of Hausdorff codimension at least $2$.
\end{proposition}

\begin{proof}
The spacetime Laplace equation may be viewed as a linear equation $\Delta u= \langle X,\nabla u\rangle$, where $X=-nf\frac{\nabla u}{|\nabla u|}$ whenever $\nabla u\ne0$ and $X=0$ when $\nabla u=0$. 
Since $X$ is $L^\infty$, the result follows immediately from \cite[Theorem 1.1]{NaberV}.
\end{proof}

The importance of spacetime harmonic functions rests to a large extent on a fundamental integral inequality that they satisfy, which here will be specialized to dimension 3. The next result follows directly from \cite[Proposition 3.2]{HKK},
by setting $k=fg$ so that $\mu=\frac{1}{2}R+3f^2$ and $J=-2\nabla f$. In this setting, the \textit{spacetime Hessian} is given by
\begin{equation}
{\bar{\nabla}}^2u:=\nabla^2u+|\nabla u|fg.
\end{equation}
Note that the spacetime Laplacian arises as the trace of this spacetime Hessian.

\begin{lemma}\label{integralformula}
Let $(M^3,\partial_\pm M^3,g)$ be a $3$-dimensional Riemannian band, and let $f\in \mathrm{Lip}(M^3)$. If $u\in C^{2,\varsigma}(M^3)$, $\varsigma\in(0,1)$ solves boundary value problem \eqref{e:bandspacetimeharmoniceq1}, then
\begin{align}\label{integralformula2}
\begin{split}
&\int_{\partial_-M^3}2|\nabla u|(2f-H)dA -\int_{\partial_+M^3}2|\nabla u|(2f+H)dA\\
\geq&\int_{M^3}\left(\frac{|\bar{\nabla}^2u|^2}{|\nabla u|}+(R+6f^2)|\nabla u|-4\langle\nabla f,\nabla u\rangle\right)dV -\int_{c_-}^{c_+} 4\pi\chi(\Sigma_t)dt
\end{split}
\end{align}
where 
$H$ is the outward mean curvature of $\partial M^3$, and $\chi(\Sigma_t)$ is the Euler characteristic of regular level sets $\Sigma_t:=u^{-1}(t)$.
\end{lemma}

Even though the function $f$ is only Lipschitz, the appearance of $\nabla f$ in \eqref{integralformula2} is justified
by Rademacher's Theorem, which ensures that the derivative exists almost everywhere. Furthermore, the Euler characteristic integrand is in fact a measurable function, which may be seen as follows. Observe that as explained in \cite[Remark 3.3]{HKK}, the conclusion of Sard's theorem still holds for $u$ even though it may not be $C^3$-smooth. 
Moreover, $u$ is a proper map and so its regular values form an open set of full measure. 
Hence, if $t_0$ is a regular value of $u$, then the function $t\mapsto\chi(\Sigma_t)$ is constant for all levels $t$ near $t_0$. We then have that $\chi(\Sigma_t)$ is continuous almost everywhere, and is therefore measurable. For more information concerning spacetime harmonic functions, we refer to the survey article \cite{BHKKZ}.

\subsection{Proof of Theorem \ref{sh}: the inequality}

Let $\Sigma^2$ be the closed surface that separates $M^3$ into two connected components $M^3_\pm$, where $E_\pm$ is contained in $M^3_\pm$. Set $w_\pm=\min\{d(E_{\pm},\Sigma^2),\frac{\pi}{\alpha}\}$, and suppose that  $w_-+w_+ \geq\frac{\pi}{\alpha}$. Consider the signed distance function $r(x)=\pm d(x,\Sigma^2)$ for $x\in M^3_\pm$. For $\varepsilon>0$ small, define the band $({\widetilde{ M^3_\varepsilon}},\partial_\pm \widetilde{ M_\varepsilon^3},g)$ by
\begin{equation}
   \widetilde{ M_\varepsilon^3}=
    \{x\in M^3\mid r(x)\in[-w_-+\varepsilon,w_+-\varepsilon]\},
\end{equation}
where the assignment $\partial_\pm \widetilde{M^3_\varepsilon}$ respects $E_\pm$.
According to \cite[Lemma C.2]{HKKZ}, $\widetilde{M^3_\varepsilon}$ is compact.
Next, append the compact components of $M^3\setminus \widetilde{M^3_\varepsilon}$ to $\widetilde{M_\varepsilon^3}$, and denote the resulting manifold by $\widehat{M^3_\varepsilon}$. Notice that each component of $M^3\setminus \widehat{M^3_\varepsilon}$ contains at least one end. By appealing to the long exact sequence of the pair $(M^3,\widehat{M_\varepsilon^3})$, and using the fact that the top homology group of an open manifold is trivial, we find that the inclusion $H_2(\widehat{M^3_\varepsilon};\mathbb{Z})\to H_2(M^3;\mathbb{Z})$ is injective. It follows that there are no spherical classes in $H_2(\widehat{M^3_\varepsilon};\mathbb{Z})$, since this property is assumed for $M^3$. Moreover, because $\Sigma^2$ separates the nonempty collections $E_\pm$, we have that at least one component of each $\partial_\pm \widetilde{M^3_\varepsilon}$ remains in $\partial \widehat{M^3_\varepsilon}$, and that the distance within $\widehat{M^3_\varepsilon}$ from $\Sigma^2$ to these components is unchanged. As in the proof of \cite[Main Theorem A]{HKKZ}, there is a small perturbation of $\widehat{M_\varepsilon^3}$ to a band $(M^3_\varepsilon,\partial_\pm M^3_\varepsilon,g)$ with smooth boundary, no spherical homology, and width at least $w_-+w_+ -3\varepsilon$. We will proceed to work with the bands $M^3_{\varepsilon}$, eventually taking a limit as $\varepsilon\rightarrow 0$ to obtain an integral inequality for a nontrivial spacetime harmonic function, which will lead to a contradiction if $w_- +w_+>\frac{\pi}{\alpha}$.

Let $u_{\varepsilon,i}$ be the spacetime harmonic function guaranteed by Proposition \ref{p:existence} satisfying
\begin{equation}\label{e:bandspacetimeharmoniceq}
\begin{cases}
    \Delta u_{\varepsilon,i}+3f_{\varepsilon,i}|\nabla u_{\varepsilon,i}|=0&\text{ in }M^3_\varepsilon ,\\
    u_{\varepsilon,i}=\pm1 &\text{ on }\partial_\pm M^3_\varepsilon ,
\end{cases}
\end{equation}
where $f_{\varepsilon,i}$ is defined in \eqref{f vi} below. 
Denote $w^\pm_\varepsilon=d(\partial_\pm M^3_\varepsilon ,\Sigma^2)$, and observe that $w^\pm_\varepsilon\ge w_\pm-2\varepsilon$ along with $w^+_\varepsilon+w^-_\varepsilon\ge \frac{\pi}{\alpha}-3\varepsilon$.
Let $h(t)$ be a Lipschitz cut-off function such that $h(t)=0$ if $t\le 0$, $h(t)=t$ if $t\in[0,\frac{\pi}{\alpha}]$, and $h(t)=\frac{\pi}{\alpha}$ if $t\ge \frac{\pi}{\alpha}$. For all large positive integers $i$ we define
\begin{equation} \label{f vi}
    f_{\varepsilon,i}(x)=
    \begin{cases}
    (1+\frac{1}{i})\frac{2\alpha}{3}\tan \left(\alpha h(r(x)+w_--2\varepsilon)-\frac{\pi}{2}+\frac{1}{i}\right) &\textnormal{   if } r(x)\le \min\{ -\frac{w^-_\varepsilon}{2},-w_\varepsilon^{-}+\frac{\pi}{6\alpha}\}
    \\
     (1+\frac{1}{i})\frac{2\alpha}{3}\tan \left(\frac{\pi}{2}-\alpha h(w_+-2\varepsilon-r(x))-\frac{1}{i}\right) &\textnormal{   if } r(x)\ge \max\{\frac{w^+_\varepsilon}{2},w^+_\varepsilon-\frac{\pi}{6\alpha}\}
     \\ (1+\frac{1}{i})\frac{2\alpha}{3}\tan \left(\mathbf{l}_{\varepsilon,i}(r(x))\right) &\textnormal{    otherwise}
    \end{cases},
\end{equation}
where $\mathbf{l}_{\varepsilon,i}(r)$ are linear functions chosen to ensure that $f_{\varepsilon,i}$ is Lipschitz. Since $w^+_\varepsilon+w^-_\varepsilon\ge\frac{\pi}{\alpha}-3\varepsilon$, an elementary but tedious calculation shows that the slope of $\mathbf{l}_{\varepsilon,i}$ is positive and less than $\alpha(1+\Tilde{C}\varepsilon)$, where the constant $\tilde{C}>0$ is independent of $\varepsilon$ and $i$. Let $\Omega_\varepsilon$ be the region defined by the third case in \eqref{f vi}. Then outside of a set of measure zero we have
\begin{align}
\label{f varepsilon i}
\begin{split}
    &\frac{4\alpha^2}{9}+f_{\varepsilon,i}^2-\frac{2}{3}|\nabla f_{\varepsilon,i}|\ge0,\quad \textnormal{\; on \;} M^3_\varepsilon\setminus \Omega_\varepsilon,
    \\ &\frac{4\alpha^2}{9}+f_{\varepsilon,i}^2-\frac{2}{3}|\nabla f_{\varepsilon,i}|\ge-C_0(\varepsilon+i^{-1}), \quad\textnormal{\; on \;} \Omega_\varepsilon ,
\end{split}
\end{align}
where $C_0$ is a constant independent of $\varepsilon$ and $i$. Note that $f_{\varepsilon,i}\to \pm\infty$ on $\partial_{\pm}M^3_{\varepsilon}$ as $i\rightarrow\infty$, so that $|f_{\varepsilon,i}|\ge |H_\varepsilon|$ for all $i$ large enough, where $H_\varepsilon$ is the mean curvature of $\partial M^3_\varepsilon$ with respect to the unit outer normal.

We will now apply the integral inequality of Lemma \ref{integralformula}. However, in order to obtain an
optimal estimate for $w_+ +w_-$, an additional divergence term is added 
to produce
\begin{equation}
\begin{split}\label{bdy term}
&\int_{\partial_-M_\varepsilon^3}|\nabla u_{\varepsilon,i}|\left(\frac{3(8-c^{-1})}{6-c^{-1}}f_{\varepsilon,i}-2H_\varepsilon\right)dA -\int_{\partial_+M_\varepsilon^3}|\nabla u_{\varepsilon,i}|\left(\frac{3(8-c^{-1})}{6-c^{-1}}f_{\varepsilon,i}+2H_\varepsilon\right)dA
     \\=&\int_{\partial_-M_\varepsilon^3}|\nabla u_{\varepsilon,i}|\left(4f_{\varepsilon,i}-2H_\varepsilon\right)dA -\int_{\partial_+M_\varepsilon^3}|\nabla u_{\varepsilon,i}|\left(4f_{\varepsilon,i}+2H_\varepsilon\right)dA
     \\ &-\int_{M^3_\varepsilon}\frac{c^{-1}}{6-c^{-1}}\textnormal{div}(f_{\varepsilon,i}\nabla u_{\varepsilon,i})dV
 \\ \ge & \int_{M_\varepsilon^3}\left[ \frac{|\bar{\nabla} ^2u_{\varepsilon,i}|^2}{|\nabla u_{\varepsilon,i}|}+|\nabla u_{\varepsilon,i}|(R+6f_{\varepsilon,i}^2)-4\langle \nabla u_{\varepsilon,i},\nabla f_{\varepsilon,i}\rangle \right] dV-\int^{c_+}_{c_-}4\pi \chi(\Sigma_t)dt
 \\&-\int_{M^3_\varepsilon}\frac{c^{-1}}{6-c^{-1}}\textnormal{div}(f_{\varepsilon,i}\nabla u_{\varepsilon,i})dV
 \\ 
 =&\int_{M_\varepsilon^3}\left[ \frac{|\bar{\nabla} ^2u_{\varepsilon,i}|^2}{|\nabla u_{\varepsilon,i}|}+|\nabla u_{\varepsilon,i}|\left(R+\frac{3(12-c^{-1})}{6-c^{-1}}f_{\varepsilon,i}^2\right)-\frac{3(8-c^{-1})}{6-c^{-1}}\langle \nabla u_{\varepsilon,i},\nabla f_{\varepsilon,i}\rangle \right] dV
 \\& -\int^{c_+}_{c_-}4\pi \chi(\Sigma^{\varepsilon,i}_t)dt,
\end{split}
\end{equation}
where $\{\Sigma^{\varepsilon,i}_t\}$ are the level sets of $u_{\varepsilon,i}$.
Notice that in the above inequality the Euler characteristic term is nonpositive, due to the maximum principle for spacetime harmonic functions and the property that $M^3_\varepsilon$ has no spherical classes.
Moreover, for sufficiently small $\varepsilon$ and large $i$, we may apply \eqref{f varepsilon i} while using the scalar curvature lower bound $R\ge -R_0$, for some constant $R_0 >0$, to find
\begin{align}
\label{omega}
\begin{split}
   & R+\frac{3(12-c^{-1})}{6-c^{-1}}f_{\varepsilon,i}^2-\frac{3(8-c^{-1})}{6-c^{-1}}|\nabla f_{\varepsilon,i}|
    \\ \ge& -R_0+\frac{3c^{-1}}{2(6-c^{-1})}f^2_{\varepsilon,i}-\frac{4\alpha^2}{9}\cdot\frac{9(8-c^{-1})}{2(6-c^{-1})}-1
    \\ \ge & -R_0-c^{-1}\Lambda_c-1.
\end{split}
\end{align}
Next, choose a fixed region $\Omega=\widetilde{M^3_{\varepsilon_0}}$, with $\varepsilon_0$ sufficiently small depending only on $c$, $R_0$, and $\Lambda_c$.  
Then on $M^3\setminus \Omega$ it follows that $r(x)\le-w_-+\varepsilon_0$ or $r(x)\ge w_+-\varepsilon_0$,
which guarantees that $|f_{\varepsilon,i}|$ is large enough to yield
\begin{equation}\label{omega c}
     \frac{3c^{-1}}{2(6-c^{-1})}f^2_{\varepsilon,i}\ge \frac{c^{-1}f^2_{\varepsilon,i}}{6-c^{-1}}+R_0+c^{-1}\Lambda_c+2
     \quad\text{ }\text{ on }\text{ }M^3_\varepsilon\backslash\Omega,
\end{equation}
for sufficiently large $i$. While on $\Omega$, $|f_{\varepsilon,i}|$ is uniformly bounded. Moreover, since
$f_{\varepsilon,i}$ blows-up on $\partial_\pm M^3_\varepsilon$, the boundary integrals of \eqref{bdy term} are nonpositive for large $i$. Hence, utilizing the second and third line of \eqref{omega}, as well as \eqref{omega c} produces
\begin{align}\label{dec part}
    \begin{split}
        0\ge& \int_{M_\varepsilon^3}|\nabla u_{\varepsilon,i}|\left(R+\frac{3(12-c^{-1})}{6-c^{-1}}f_{\varepsilon,i}^2-\frac{3(8-c^{-1})}{6-c^{-1}}|\nabla f_{\varepsilon,i}|\right)dV
        \\ \ge & -\int_{\Omega} (R_0+c^{-1}\Lambda_c+1)|\nabla u_{\varepsilon,i}|dV+\int_{M^3_\varepsilon\setminus\Omega}\left(\frac{c^{-1}}{6-c^{-1}}f^2_{\varepsilon,i}+1\right)|\nabla u_{\varepsilon,i}|dV.
    \end{split}
\end{align}

In order to extract a convergent subsequence, we now rescale $u_{\varepsilon,i}$ similarly to that which is done in \cite[proof of Main Theorem A]{HKKZ}, and define
\begin{equation}
    \Tilde{u}_{\varepsilon,i}(x)=\frac{u_{\varepsilon,i}(x)-\mathcal A_{\varepsilon,i}}{\sup_\Omega |\nabla u_{\varepsilon,i}|}, \quad\quad\quad \mathcal A_{\varepsilon,i}=\frac{1}{|\Omega|}\int_{\Omega}u_{\varepsilon,i}dV.
\end{equation}
The normalized function $\Tilde u_{\varepsilon,i}$ satisfies $\sup_{\Omega}|\nabla\Tilde{u}_{\varepsilon,i}|=1$,
and has vanishing average value on $\Omega$.
Therefore, \eqref{dec part} yields
\begin{equation} \label{gradient integral}
  \int_{M^3_\varepsilon\setminus\Omega}\left(\frac{c^{-1}}{6-c^{-1}}f^2_{\varepsilon,i}+1\right)|\nabla \Tilde{u}_{\varepsilon,i}|dV \le \int_{\Omega} (R_0+c^{-1}\Lambda_c+1)|\nabla \tilde{u}_{\varepsilon,i}|dV\le (R_0+c^{-1}\Lambda_c+1)|\Omega|.
\end{equation}
Since $|f_{\varepsilon,i}|$ is uniformly bounded on $\Omega$ we have
\begin{equation}\label{gradient bound}
    \int_{M^3_\varepsilon} \left(f^2_{\varepsilon,i}+1\right)|\nabla \tilde{u}_{\varepsilon,i}| dV\le \left(\frac{6-c^{-1}}{c^{-1}}+1\right)(R_0+c^{-1}\Lambda_c+1)|\Omega|+|\Omega|\sup_{\Omega} (|f_{\varepsilon,i}|^2+1)\le C_1,
\end{equation}
where $C_1$ is independent of $\varepsilon$ and $i$.
Since the average of $\tilde{u}_{\varepsilon,i}$ vanishes on $\Omega$, we may apply a version of the Poincar\'e inequality 
on $M^3_\varepsilon$ \cite[Theorem 1]{poincare} to conclude that  $\|\tilde{u}_{\varepsilon,i}\|_{W^{1,1}( M^3_\varepsilon)}$ is bounded by a constant independent of the index $i$.
Therefore, by passing to a subsequence (in $i$), $\tilde{u}_{\varepsilon,i}$ converges to a function $\tilde{u}_\varepsilon$ in $L^p(M^3_\varepsilon)$, for $p\in[1,\frac{3}{2})$, as $i\to\infty$. Because $\tilde{u}_{\varepsilon,i}$ solves the elliptic spacetime Laplace equation, uniform $L^p( M^3_\varepsilon)$ bounds for $\tilde{u}_{\varepsilon,i}$ imply uniform control in $C^{2,\varsigma}_{loc}( \widetilde{M^3_{2\varepsilon}})$, $\varsigma\in(0,1)$; here we have used the fact that $f_{\varepsilon,i}\to f_\varepsilon$ pointwise on the interior of $\widetilde{M^3_{2\varepsilon}}$ as $i\to\infty$.   Thus, $\tilde{u}_{\varepsilon,i}$ also converges subsequentially as $i\to\infty$ to $\tilde{u}_\varepsilon$ in $C^{2,\varsigma}_{loc}(\widetilde{M^3_{2\varepsilon}})$ for some $\varsigma\in(0,1)$, and the limit satisfies $\Delta \tilde{u}_\varepsilon+3f_\varepsilon|\nabla \tilde{u}_\varepsilon|=0$ on the interior of $\widetilde{M^3_{2\varepsilon}}$.

To obtain further properties of $|\nabla \tilde{u}_{\varepsilon}|$ observe that from \eqref{bdy term}, the fact that the boundary terms are nonpositive, together with \eqref{dec part} and \eqref{gradient bound} we obtain
\begin{align}
\begin{split}
    (R_0+c^{-1}\Lambda_c+1)|\Omega|
    \ge&
\int_{M^3_\varepsilon}\frac{\big|\nabla^2\tilde{u}_{\varepsilon,i}+f_{\varepsilon,i}|\nabla \tilde{u}_{\varepsilon,i}|g\big|^2}{|\nabla \tilde{u}_{\varepsilon,i}|}dV
\\=&\int_{M^3_\varepsilon}\frac{|\nabla^2\tilde{u}_{\varepsilon,i}|^2}{|\nabla \tilde{u}_{\varepsilon,i}|}-3f^2_{\varepsilon,i}|\nabla u_{\varepsilon,i}|dV
    \\ \ge&\int_{M^3_\varepsilon} 4|\nabla|\nabla \tilde{u}_{\varepsilon,i}|^\frac{1}{2}|^2dV-3C_1.
\end{split}
\end{align}
Since $|\nabla \tilde{u}_{\varepsilon,i}|^\frac{1}{2}$ is bounded in $H^1(M^3_\varepsilon)$ independent of $i$, it has a weak subsequential limit in 
$H^1(M^3_\varepsilon)$ 
which also converges strongly in $L^2(M^3_\varepsilon)$. In light of the fact that 
$f_{\varepsilon,i}$ blows-up uniformly on $M^3_{\varepsilon}\setminus\widetilde{M^3_{2\varepsilon}}$, the inequality \eqref{gradient bound} implies that the limit function $|\nabla \tilde{u}_{\varepsilon}|^{\frac{1}{2}}\equiv 0$ a.e. on this domain. Note that even though $\tilde{u}_{\varepsilon}$ may not be defined $M^3_{\varepsilon}\setminus\widetilde{M^3_{2\varepsilon}}$, with a slight abuse of notation we still denote the limit of $|\nabla \tilde{u}_{\varepsilon,i}|^{\frac{1}{2}}$ (which is defined globally on $M_{\varepsilon}^3$) in terms of $\tilde{u}_{\varepsilon}$. Moreover, for a.e. $\varepsilon$ the boundary $\partial \widetilde{M^3_{2\varepsilon}}$ is a Lipschitz submanifold \cite{Rifford}, and therefore by applying the trace theorem we find that $|\nabla\tilde{u}_{\varepsilon}|^{\frac{1}{2}}$ vanishes up to a set of measure zero on this set. It follows that $|\nabla \tilde{u}_{\varepsilon}|^{\frac{1}{2}}\in H^{1}_{0}(\widetilde{M^3_{2\varepsilon}})$ for all such $\varepsilon$; below, it will always be assumed that $\varepsilon$ satisfies this property. 

We are now ready to return to the integral inequality.
Taking the limit as $i\to \infty$ and applying Fatou's Lemma to \eqref{bdy term} yields
\begin{align}\label{e:iFatous}
     0 \ge & \int_{\widetilde{M_{2\varepsilon}^3}} \left[\frac{|\nabla ^2\tilde{u}_\varepsilon+f_\varepsilon|\nabla \tilde{u}_\varepsilon|g|^2}{|\nabla \tilde{u}_\varepsilon|}+|\nabla \tilde{u}_\varepsilon|\left(R+\frac{3(12-c^{-1})}{6-c^{-1}}f_{\varepsilon}^2\right)-\frac{3(8-c^{-1})}{6-c^{-1}}\langle \nabla \tilde{u}_{\varepsilon},\nabla f_{\varepsilon}\rangle \right] dV.
\end{align}
Consider the first two terms of \eqref{e:iFatous}. Using a Kato inequality similar to \cite[Remark 4.4]{HKKZ}, Proposition \ref{critical} which shows that the set of critical points for $\tilde{u}_{\varepsilon}$ is of measure zero, and the spectral hypothesis, we obtain
\begin{align} \label{hessian u}
\begin{split}
&\int_{\widetilde{M_{2\varepsilon}^3}} \left(\frac{|\nabla ^2\tilde{u}_\varepsilon+f_\varepsilon|\nabla \tilde{u}_\varepsilon|g|^2}{|\nabla \tilde{u}_\varepsilon|}+R|\nabla \tilde{u}_\varepsilon| \right)dV
        \\ \ge&  \int_{\widetilde{M_{2\varepsilon}^3}}\left(\frac{3|\nabla|\nabla \tilde{u}_\varepsilon|+f_\varepsilon\nabla \tilde{u}_\varepsilon|^2}{2|\nabla \tilde{u}_\varepsilon|}+ R|\nabla \tilde{u}_\varepsilon| \right) dV
        \\ =& \int_{\widetilde{M_{2\varepsilon}^3}}\left(6|\nabla|\nabla \tilde{u}_\varepsilon|^{\frac{1}{2}}|^2+6\langle\nabla|\nabla \tilde{u}_\varepsilon|^{\frac{1}{2}},f_\varepsilon|\nabla \tilde{u}_\varepsilon|^{-\frac{1}{2}}\nabla \tilde{u}_\varepsilon\rangle+ \left(R+\frac{3}{2}f_\varepsilon^2\right)|\nabla \tilde{u}_\varepsilon|\right)dV
        \\  =&\int_{\widetilde{M_{2\varepsilon}^3}}\left((6-c^{-1})\left|\nabla|\nabla \tilde{u}_\varepsilon|^\frac{1}{2}+\frac{3}{6-c^{-1}}f_\varepsilon|\nabla \tilde{u}_\varepsilon|^{-\frac{1}{2}}\nabla \tilde{u}_\varepsilon\right|^2+c^{-1}|\nabla|\nabla \tilde{u}_\varepsilon|^\frac{1}{2}|^2 \right.
        \\&\left.+R|\nabla \tilde{u}_\varepsilon|+\left(\frac{3}{2}-\frac{9}{6-c^{-1}}\right)f_\varepsilon^2|\nabla \tilde{u}_\varepsilon|\right)dV
        \\ \ge & 
      \int_{\widetilde{M_{2\varepsilon}^3}}\left(c^{-1}\Lambda_c|\nabla \tilde{u}_\varepsilon|+\left(\frac{3}{2}-\frac{9}{6-c^{-1}}\right)f_\varepsilon^2|\nabla \tilde{u}_\varepsilon|\right)dV.
\end{split}
\end{align}
Combining \eqref{e:iFatous} and \eqref{hessian u} then produces
\begin{align}\label{e:sthepsilonmass}
\begin{split}
 0  \ge&\int_{\widetilde{M^3_{2\varepsilon}}}\left[\left(c^{-1}\Lambda_c+\frac{9(8-c^{-1})}{2(6-c^{-1})}f_\varepsilon^2\right)|\nabla \tilde{u}_\varepsilon|-\frac{3(8-c^{-1})}{6-c^{-1}}\langle\nabla \tilde{u}_\varepsilon,\nabla f_\varepsilon\rangle \right] dV 
 \\ \ge& \int_{\widetilde{M_{2\varepsilon}^3}}c^{-1}\Lambda_c\left(1+\frac{9}{4\alpha^2}f_\varepsilon^2-\frac{3}{2\alpha^2}|\nabla f_\varepsilon|\right)|\nabla \tilde{u}_\varepsilon|dV,
\end{split}
\end{align}
where $\alpha=\sqrt{\frac{\Lambda_c(6-c^{-1})}{2c(8-c^{-1})}}$.

To proceed, we shall inspect the limit as $\varepsilon\rightarrow 0$. 
First note that applying Fatou's lemma to equation \eqref{gradient bound} yields
\begin{equation}
    \int_{\widetilde{M_{2\varepsilon}^3}}|\nabla \tilde{u}_\varepsilon|dV\le \liminf_{i\to\infty}\int_{\widetilde{M_{2\varepsilon}^3}}|\nabla \tilde{u}_{\varepsilon,i}|dV\le C_1.
\end{equation}
Moreover, since $\tilde{u}_{\varepsilon,i}$ has vanishing average on $\Omega$, the same is true of $\tilde{u}_{\varepsilon}$, and thus
utilizing again a version of the Poincar\'e inequality we obtain uniform $W^{1,1}(\widetilde{M^3_{2\varepsilon}})$ bounds for $\tilde{u}_\varepsilon$. 
By passing to a subsequence, $\tilde{u}_\varepsilon\to u$ in $L^{p}_{loc}(\bar{M}^3)$ for any $p\in[1,\frac{3}{2})$, where
\begin{equation}
\Bar{M}^3=\cup_\varepsilon\widetilde{M_{2\varepsilon}^3}=\{x\in M^3 \mid r(x)\in(-w_-,w_+)\}.     
\end{equation}
As before, since $\tilde{u}_{\varepsilon}$ satisfies the elliptic spacetime Laplacian, we may boot-strap to find subsequential convergence $\tilde{u}_{\varepsilon}\rightarrow u$ in $C^{2,\varsigma}_{loc}(\bar{M}^3)$, for some $\varsigma\in(0,1)$. Furthermore, $\Delta u+3f|\nabla u|=0$ on $\bar{M}^3$ with
\begin{equation}\label{aeoihnfoinqhg}
    f(x)=\begin{cases}
    \frac{2\alpha}{3}\tan \left(\alpha h(r(x)+w_-)-\frac{\pi}{2}\right) &\textnormal{   if } r(x)\le \min\{ -\frac{w_-}{2},-w_{-}+\frac{\pi}{6\alpha}\}
    \\ \frac{2\alpha}{3}\tan \left(\frac{\pi}{2}-\alpha h(w_+-r(x))\right)& \textnormal{   if } r(x)\ge \max\{ \frac{w_+}{2},w_+-\frac{\pi}{6\alpha}\}
    \\ \frac{2\alpha}{3}\tan \left(\mathbf{l}(r(x))\right)& \textnormal{  otherwise}
    \end{cases}
\end{equation}
since  $\lim_{\varepsilon\to0}w^\pm_\varepsilon=w_\pm$, 
where $\mathbf{l}(r)$ is a linear function which ensures that $f$ is Lipschitz.

Finally, if $w_-+w_+>\frac{\pi}{\alpha}$ then the slope of $\mathbf{l}$ would be strictly less than $\alpha$ in some region of nonzero measure, which produces
\begin{equation} \label{ode f}
  1+\frac{9}{4\alpha^2}f^2-\frac{3}{2\alpha^2}|\nabla f|>1+ \tan^2(\mathbf{l}(r(x)))-\sec^2(\mathbf{l}(r(x)))=0.
\end{equation}
Furthermore, taking the limit of \eqref{e:sthepsilonmass} with Fatou's lemma implies that
\begin{equation}
 0\ge\int_{\bar{M}^3}
c^{-1}\Lambda_c\left(1+\frac{9}{4\alpha^2}f^2-\frac{3}{2\alpha^2}|\nabla f|\right)|\nabla u|
dV .  
\end{equation}
Since $u$ is nontrivial as $\sup_{\Omega}|\nabla u|=1$, Proposition \ref{critical} shows that $|\nabla u|$ can only vanishes on a set of measure zero, and therefore a contradiction is obtained. We then have that $w_- +w_+ \leq \frac{\pi}{\alpha}$, from which the desired conclusion follows.

\subsection{Proof of Theorem \ref{sh}: the case of equality} We now assume that equality holds in \eqref{rad 3}. Since 
$w_- +w_+ \leq \frac{\pi}{\alpha}$, neither one of $w_{\pm}$ can be $\frac{\pi}{\alpha}$ as both $d(E_{\pm},\Sigma^2)$ must be positive. It follows that $w_- +w_+= \frac{\pi}{\alpha}$. Therefore, since $r(x)\in (-w_-,w_+)$ on $\bar{M}^3$, in this region \eqref{aeoihnfoinqhg} gives
\begin{equation} 
f(x)=\frac{2\alpha}{3}\tan \left(\alpha r(x)+\alpha w_- -\frac{\pi}{2}\right)=-\frac{2\alpha}{3}\cot(\alpha\rho(x)),
\end{equation}
where $\rho(x)=r(x)+w_-$.   
By inspecting \eqref{e:iFatous}-\eqref{e:sthepsilonmass}, using Fatou's lemma, and Proposition \ref{critical} we find that 
\begin{align}
\label{nabla u}
    \nabla |\nabla u|^\frac{1}{2}+\frac{3}{6-c^{-1}}f|\nabla u|^{-\frac{1}{2}}\nabla u=0
\end{align}
holds almost everywhere. Then integrating this equation along curves emanating from regular points for $u$, shows that in fact $|\nabla u|\neq0$ holds globally on $\bar{M}^3$.

Let  $\{e_1,e_2,e_3=\frac{\nabla u}{|\nabla u|}\}$ be an orthonormal frame.  From  \eqref{e:sthepsilonmass}, we deduce that $\nabla u$ is a multiple of $\nabla f$, and therefore $e_3=\nabla \rho$. This implies that $u$ is a function of $\rho$.  Furthermore, 
the first two lines of \eqref{hessian u} combined with \cite[Remark 4.4]{HKKZ} show that
\begin{equation}
    \nabla_{ij} u=0 \textnormal{ \;if }\; i\neq j,\quad\quad\quad \nabla_{11} u=\nabla_{22} u.
\end{equation}
Thus, $(\bar{M}^3,g)$ is a warped product with $g=d\rho^2+\phi^2(\rho)g_0$, where $g_0$ is a metric on $\Sigma^2$ and $\phi$ is a positive continuously differentiable function on $(0,\frac{\pi}{\alpha})$. 
Next observe that inserting $e_3$ into \eqref{nabla u} yields
\begin{equation}\label{gradient u}
    \nabla_{3} |\nabla u|=-\frac{6}{6-c^{-1}}f|\nabla u|, 
\end{equation}
which implies that up to a scaling constant we have 
\begin{equation}\label{aoihofinaoh}
  |\nabla u|(\rho)=[\sin\left(\alpha \rho\right)]^\frac{4c}{6c-1}.  
\end{equation}
Moreover, using the spacetime harmonic equation combined with \eqref{gradient u} produces
\begin{equation} \label{mean curv}
\frac{2\phi_\rho}{\phi}=H=\frac{\Delta u-\nabla_{33}u}{|\nabla u |}=\frac{3c^{-1}-12}{6-c^{-1}}f,
\end{equation} 
where $H$ is the mean curvature of level sets with respect to $e_3$.
Hence, it follows that up to scaling $\phi(\rho)=[\sin(\alpha \rho)]^\frac{4c-1}{6c-1}$. 
When $c\neq \frac{1}{4}$, due to the behavior of $\phi$ at the ends we find that $\Bar{M}^3$ cannot be strictly contained in a connected open manifold, and thus $\bar{M}^3=M^3$. 
If $c=\frac{1}{4}$, then $\phi=1$ and $\Bar{M}^3$ is a cylinder. After taking the limit in \eqref{hessian u}, we conclude that the function $|\nabla u|^\frac{1}{2}=\sin(\alpha \rho)$ minimizes the Rayleigh quotient \eqref{variation}. It follows that the $c$-spectral constant of $\bar{M}^3$ must be $\Lambda_c$. If $\bar{M}^3$ was properly contained in $M^3$, then its $c$-spectral constant would be strictly larger than that of $M^3$, which is also $\Lambda_c$. Therefore $\bar M^3=M^3$, in this case as well. See Figure \ref{equality}, for a depiction of the different types of behavior for the warped product according to the value of $c$.



\begin{figure}[H]
\includegraphics[totalheight=2.5cm]{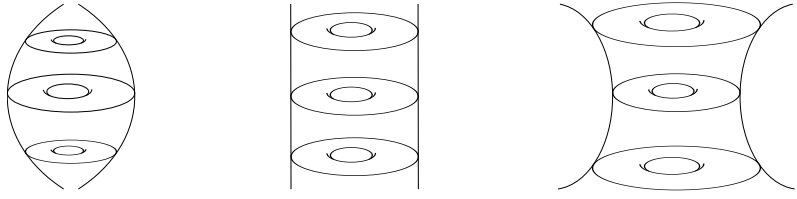}
\caption{Case of equality model geometries from left to right: $c>\frac14$, $c=\frac14$, $c<\frac14$.}
\label{equality}
\end{figure}


It remains to show that $(\Sigma^2 ,g_0)$ is a flat torus. According to \cite[Corollary 43]{Neill} we have
\begin{equation}
    \Ric(\partial_\rho,\partial_\rho)=-\frac{2\phi_{\rho\rho}}{\phi}.
\end{equation}
Moreover taking two traces of the Gauss equations, denoting the Gaussian curvature of $g_0$ by $K_0$, and noting that the second fundamental form of $(\Sigma^2, g_0)\hookrightarrow (M^3,g)$ is given by $II=\phi\phi_\rho g_0$ with mean curvature $H=2\phi^{-1}\phi_{\rho}$, yields
\begin{align}
\begin{split}
R=&2\Ric(\partial_\rho,\partial_\rho)+2\phi^{-2}K_0+|II|^2-H^2
    \\=&-2\phi^{-2}\left(\phi_\rho^2+2\phi\phi_{\rho\rho}\right)+2\phi^{-2}K_0.
    \end{split}
\end{align}
Let $L=-\Delta+cR-\Lambda_c$, then a tedious but elementary calculation using the explicit expressions for $|\nabla u|^{\frac{1}{2}}$ and $\phi$ along with the relation between $\alpha$ and $\Lambda_c$ in \eqref{rad 3}, shows that 
\begin{equation} 
\begin{split} \label{eigenfunction}
     L(|\nabla u|^{\frac{1}{2}})=&-\partial_{\rho\rho}(|\nabla u|^\frac{1}{2})-H\partial_\rho(|\nabla u|^\frac{1}{2})+(c R-\Lambda_c)|\nabla u|^\frac{1}{2}
    \\=& -\partial_{\rho\rho}(|\nabla u|^\frac{1}{2})-\frac{2\phi_\rho}{\phi}\partial_{\rho}(|\nabla u|^\frac{1}{2})-\left(\frac{2c}{\phi^2}(\phi_\rho^2+2\phi\phi_{\rho\rho})-2c\phi^{-2}K_0+\Lambda_c\right)|\nabla u|^\frac{1}{2}
    \\=&2cK_0\phi^{-2}|\nabla u|^\frac{1}{2}.
\end{split}
\end{equation} 
Again using the explicit expressions for function and metric, it may be verified that $|\nabla u|^{\frac{1}{2}}\in H^1_0(M^3)$ for $c>\frac{1}{6}$. Furthermore, as observed at the end of the previous paragraph, the Rayleigh quotient evaluated at $|\nabla u|^{\frac{1}{2}}$ agrees with $\Lambda_c$, and thus $L(|\nabla u|^{\frac{1}{2}})=0$. Hence $K_0 =0$, and $(\Sigma^2,g_0)$ is a flat torus. This completes the `only if' direction in the case of equality statement.

To verify the `if' direction, it must be shown that given $\alpha>0$, $c>\tfrac16$, and a flat torus $(T^2,g_0)$, the $c$-spectral constant $\Lambda_c$ of the warped product $(M^3,g)=((0,\tfrac{\pi}{\alpha})\times T^2,d\rho^2+\phi^2(\rho)g_0)$ agrees with $\Lambda'_c=\frac{2c(8-c^{-1})\alpha^2}{6-c^{-1}}$.
According to \eqref{eigenfunction}, the function $u_0=|\nabla u|^\frac{1}{2}$ given by \eqref{aoihofinaoh} satisfies 
\begin{equation}
(-\Delta+cR)u_0=\Lambda'_c u_0,
\end{equation}
and $u_0>0$ on $M^3$. This implies that $\Lambda_c \leq \Lambda'_c$. If $\Lambda_c < \Lambda'_c$, then there exists a test function with compact support having Rayleigh quotient strictly less than $\Lambda'_c$. We can therefore find a smooth manifold with boundary $\widetilde{M^3_{\varepsilon}}$ containing this support, and conclude that the principal eigenvalue $\lambda$ for this domain with Dirichlet boundary conditions is strictly less than $\Lambda'_c$. Let $\hat{u}>0$ be the corresponding principal eigenfunction for this domain, so that
\begin{equation}
    (-\Delta+cR)\hat{u}=\lambda \hat{u} \quad\text{ }\text{ in }\text{ }\widetilde{M^3_{\varepsilon}},\quad\quad\quad\quad\quad
    \hat{u}=0\quad\text{ }\text{ on }\text{ }\partial\widetilde{M^3_{\varepsilon}}.
\end{equation}
Observe that
\begin{equation}
\begin{split}
     \Delta (\hat{u} u_0^{-1})=&u_0^{-1}\Delta\hat{u}+2\nabla \hat{u}\cdot \nabla u_0^{-1}+\hat{u}\Delta u_0^{-1}
    \\=&u_0^{-1}(cR-\lambda) \hat{u}-2\nabla(\hat{u}u_0^{-1})\cdot \nabla \log u_0-2\hat{u}u_0^{-3}|\nabla u_0|^2-\hat{u}u_0^{-2}\Delta u_0+2\hat{u}u_0^{-3}|\nabla u_0|^2
    \\=&(\Lambda'_c-\lambda)\hat{u}u_0^{-1}-2\nabla(\hat{u}u_0^{-1})\cdot \nabla \log u_0.
\end{split}
\end{equation}
Thus, by applying the maximal principal to $\hat{u}u_0^{-1}$ on $\widetilde{M^3_{\varepsilon}}$, we see that this function must vanish, which is a contradiction. Hence $\Lambda_c =\Lambda'_c$.


\section{The Spinorial Callias Operator Approach}
\label{sec3}

The purpose of this section is to establish Theorem \ref{T:spinors}. It will be assumed in what follows that $n>1$, as the inequality \eqref{aoiwfoinah} for $n=1$ is trivially satisfied. 
Before beginning the proof, we will first introduce the requisite machinery and notation. 

\subsection{Background}
The following fact describes the fundamental property, from the perspective of this work, of bands that admit the $\hat{A}$-overtorical condition. It is well known, see \cite[Example 7.5]{CecchiniZeidler1}.

\begin{proposition}\label{p:ahatarea}
Suppose that $(M^n,\partial_\pm M^n,g)$ is an odd dimensional Riemannian spin band which is $\hat{A}$-overtorical. For any $\delta>0$, there exists a Hermitian bundle $\mathcal{E}$ over $M^n$ with a metric compatible connection $\nabla^{\mathcal{E}}$ such that
\begin{enumerate}
    \item the curvature $R^{\mathcal{E}}$ of $(\mathcal{E},\nabla^{\mathcal{E}})$ satisfies $|R^{\mathcal{E}}|<\delta$,
    \item the wedge product of the $\mathbf{\widehat{A}}$ form of $\partial_-M^n$ with the Chern character of $\mathcal{E}|_{\partial_- M^n}$ satisfies
    \begin{equation}\label{aonhfoianhg}
        \int_{\partial_-M^n} \mathbf{\widehat{A}}(\partial_-M^n)\wedge \mathrm{ch}(\mathcal{E}|_{\partial_-M^n})\neq 0.
    \end{equation}
\end{enumerate}
\end{proposition}


The next task is to introduce the relevant bundles and structure required to describe the spinors used to prove Theorem \ref{T:spinors}. We will closely follow the exposition in \cite[Sections 2 and 3]{CecchiniZeidler1}. Consider an odd dimensional Riemannian band $(M^n,\partial_\pm M^n,g)$ with a spin structure. Let $S'\to M^n$ denote the associated complex spinor bundle,
equipped with the connection induced by the Levi-Civita connection. Given a Hermitian bundle $\mathcal{E}\to M^n$ with a metric connection, consider the bundle $S=(S'\otimes \mathcal{E})\oplus(S'\otimes \mathcal{E})=:S^-\oplus S^+$. This bundle may be equipped with an action of the Clifford algebra which interchanges its summands according to the formula 
\begin{equation}
    v \text{ }\!\cdot =\left(
    \begin{array}{cc}
    0&v\cdot\otimes \text{ }\!\mathbb{I}_{\mathcal{E}}\\
    v\cdot\otimes \text{ }\!\mathbb{I}_{\mathcal{E}}&0
    \end{array}
    \right),
\end{equation}
where $v\in TM^n$ is a vector and $\mathbb{I}_{\mathcal{E}}$ denotes the identity on $\mathcal{E}$. Here and throughout, we use the sign convention $v\cdot w+w\cdot v=-2g(v, w)$ for vectors $v,w$. The bundle $S$ also carries a natural involution $\sigma$ defined by 
\begin{equation}\label{e:involution}
    \sigma=\left(
    \begin{array}{cc}
    0&-i\\
    i&0
    \end{array}
    \right),
\end{equation}
where we are implicitly making reference to the direct sum description of $S$. In a standard manner, $S$ inherits a connection from $M^n$ and $\mathcal{E}$, and one may form the corresponding Dirac operator $\slashed{\partial}$. Given a Lipschitz function $f$ on $M^n$, we may also consider the {\emph{Callias operator}}
\begin{equation}
\mathcal{B}_{f}\varphi=\slashed\partial\varphi+f\sigma\varphi.
\end{equation}
Notice that there is a decomposition $\mathcal{B}_f=\mathcal{B}_f^+\oplus \mathcal{B}_f^-$ where $\mathcal{B}_f^\pm$ maps $S^\pm$ to $S^\mp$.

Appropriate boundary conditions are required to set up an elliptic boundary value problem, namely we will consider
\begin{equation}\label{e:spinorpde}
    \begin{cases}
        {\mathcal{B}_f}\varphi=0 &\text{ in }M^n\\
        \mp\nu\cdot \sigma\varphi=\varphi & \text{ on }\partial_\pm M^n
    \end{cases}
\end{equation}
where $\nu$ denotes the unit outward normal to $\partial M^n$. 
This yields elliptic boundary value problems associated to $\mathcal{B}^\pm_f$ which are adjoint to each other, and therefore 
\begin{equation}
\mathrm{Index}(\mathcal{B}^\pm_f)=\mathrm{dim}(\mathrm{ker}(\mathcal{B}^\pm_f))-\mathrm{dim}(\mathrm{ker}(\mathcal{B}^\mp_f)). 
\end{equation}
According to \cite[Corollary 3.10]{CecchiniZeidler1}, the expression in \eqref{aonhfoianhg} is the Fredholm index of the boundary value problem associated to $\mathcal{B}^{\pm}_f$. Hence, if $(M^n,\partial_{\pm} M^n,g)$ is $\hat{A}$-overtorical, Proposition \ref{p:ahatarea} implies that there is a source of Hermitian bundles $\mathcal{E}$ such that there are nontrivial solutions to this boundary value problem.

\begin{proposition}\label{p:spinorexistence}
Suppose that $(M^n,\partial_\pm M^n,g)$ is an odd dimensional Riemannian spin band which is $\hat{A}$-overtorical. Let $\mathcal{E}$ be a Hermitian bundle given by Proposition \ref{p:ahatarea}. Then for any Lipschitz function $f$ and $\varsigma\in(0,1)$, there exists a nontrivial $C^{1,\varsigma}$ solution to \eqref{e:spinorpde}.
\end{proposition}



\subsection{Proof of Theorem \ref{T:spinors}}
Let $f$ be a Lipschitz function on $M^n$ such that $f$ is positive on $\partial_+M^n$ and negative on $\partial_- M^n$, to be determined later. Proposition \ref{p:spinorexistence} yields a nontrivial spinor $\varphi$ satisfying \eqref{e:spinorpde}.
In order to express the associated B{\"o}chner-Lichnerowitz-Weitzenbock formula, let $\mathcal{P}$ denote the Penrose operator acting on spinors according to the formula
\begin{equation}
    \mathcal{P}_X\varphi=\nabla_X\varphi-\frac{1}{n}X\cdot \slashed{\partial}\varphi,
\end{equation}
for any vector field $X$. Fix $p\in M^n$ and let $\{e_l\}_{l=1}^n$ be an orthonormal basis at $p$. Consider the quantities $v_l=e_l\cdot\nabla_{e_l} \varphi+\frac{f}{n} \sigma \varphi$, and note that according to equation \eqref{e:spinorpde} we have $\sum_{l=l}^{n} v_l=0$. Write $v=(v_1,\bar v)$ and observe that by Cauchy-Schwarz $(n-1)|\bar{v}|^2\ge|v_1|^2$. Thus $|v|^2=v_1^2+|\bar v|^2\geq \frac{n}{n-1}|v_1|^2$, and we arrive at the following Kato-type inequality
\begin{equation}
|\mathcal{P}\varphi|^2=\sum_{l=1}^{n}\left|\nabla_l \varphi-\frac{f}{n}e_l\cdot \sigma \varphi\right|^2
    \geq \frac{n}{n-1}\left|\nabla_{1} \varphi-\frac{f}{n}e_1\cdot \sigma \varphi\right|^2.
\end{equation}
By expanding the right-hand side and denoting $\beta=\frac{n}{n-1}-\frac{1}{4c}>0$, it follows that
\begin{equation}\label{e:spinorkato}
    \begin{split}
    |\mathcal{P}\varphi|^2\geq&
     \frac{n}{n-1}|\nabla_{1} \varphi|^2-\frac{2}{n-1}f\langle \nabla_1 \varphi, e_1\cdot \sigma \varphi\rangle
   +\frac{1}{n(n-1)}f^2|\varphi|^2
   \\=& \frac{1}{4c}|\nabla_{1}\varphi|^2+\beta\left|\nabla_{1} \varphi-\frac{1}{\beta(n-1)}fe_1\cdot \sigma \varphi\right|^2+\left(\frac{1}{n(n-1)}-\frac{1}{\beta(n-1)^2}\right)f^2|\varphi|^2
   \\ \ge& \frac{1}{4c}|\nabla_1 \varphi|^2+\underbrace{\left(\frac{1}{n(n-1)}-\frac{1}{\beta(n-1)^2}\right)}_{\beta_1}f^2|\varphi|^2.
   \end{split}
\end{equation}
Since $|\varphi|$ is Lipschitz, Radamacher's Theorem ensures that it is differentiable almost everywhere. Now if $\nabla| \varphi|\neq0$ at $p$, then we may choose a basis with $e_1$ given by the unit gradient so that $|\nabla_{1}|\varphi||=|\nabla |\varphi||$, whereas if $\nabla|\varphi|=0$ at $p$ then this equality holds trivially for any choice of $e_1$. Thus, \eqref{e:spinorkato} implies that
\begin{equation}\label{e:spinorkato2}
    |\mathcal{P} \varphi|^2\geq\frac{1}{4c}|\nabla |\varphi||^2+\beta_1 f^2|\varphi|^2
\end{equation}
holds almost everyhwere.

According to \cite[Proposition 4.2]{CecchiniZeidler1} and the proof of \cite[Theorem 4.3]{CecchiniZeidler1}, we have as a consequence of the B{\"o}chner-Lichnerowitz-Weitzenbock formula that
\begin{equation}
    \begin{split}
       &\int_{\partial_-M ^n} (f-\frac{n}{2(n-1)}H)|\varphi|^2dA-\int_{\partial_+M^n}(f+\frac{n}{2(n-1)}H)|\varphi|^2dA
        \\ = &
 \int_{M^n}\left( \frac{n}{n-1}\left(|\mathcal{P} \varphi|^2+\langle \varphi,\frac{R}{4} \varphi+\mathcal{R}^{\mathcal{E}} \varphi\rangle\right)+
 \langle \varphi,f^2 \varphi+\nabla f\cdot\sigma \varphi\rangle \right)dV
   \\ \ge & \int_{M^n}\left( \frac{n}{n-1}\left(\frac{1}{4c}|\nabla |\varphi||^2+\beta_1f^2|\varphi|^2+\frac{R}{4}|\varphi|^2-\gamma_n|R^{\mathcal{E}}||\varphi|^2 \right)
   +\langle \varphi,f^2 \varphi+\nabla f \cdot \sigma \varphi\rangle \right)
   dV,
   \label{main spinor}
    \end{split}
\end{equation}
where $\mathcal{R}^{\mathcal{E}}$ is the $\mathcal{E}$-curvature acting on sections of $S$ and $\gamma_n$ is a dimensional constant encountered when applying Cauchy-Schwarz to $\langle\varphi,\mathcal{R}^{\mathcal{E}}\varphi\rangle$. 
To continue, notice that from an integration by parts the following identity holds
\begin{equation}
\begin{split}
    \int_{\partial M^n}\langle \nu\cdot\varphi,f\sigma \varphi\rangle dA
 =&\int_{M^n}\left(\langle \slashed{\partial} \varphi,f\sigma \varphi \rangle-
    \langle  \varphi,\slashed{\partial}f\sigma \varphi \rangle\right) dV
    \\=&-\int_{M^n}\left(2f^2|\varphi|^2+\langle \varphi,\nabla f\cdot\sigma \varphi\rangle\right) dV,
\end{split}
\label{divergence}
\end{equation}
where we have made use of the fact that $\sigma X\cdot=-X\cdot\sigma$ for vector fields $X$. Leveraging the boundary condition for $\varphi$, one may multiply \eqref{divergence} by $\frac{n}{n-1}\beta_1$ and sum the result with \eqref{main spinor} to obtain
\begin{equation}
    \begin{split}
    \label{n beta1}
        &\int_{\partial_-M ^n} \left[\left(1-\frac{n\beta_1}{n-1}\right)f-\frac{n}{2(n-1)}H\right]|\varphi|^2dA
        \\&-\int_{\partial_+M^n}\left[\left(1-\frac{n\beta_1}{n-1}\right)f+\frac{n}{2(n-1)}H\right]|\varphi|^2dA
        \\ \ge&\int_{M^n} \frac{n}{n-1}\left(\frac{1}{4c} |\nabla|\varphi||^2+\frac{R}{4}|\varphi|^2-\gamma_n|R^{E}||\varphi|^2\right)dV
 \\&  +\int_{M^n}\left\langle \varphi, \underbrace{\left(1-\frac{n\beta_1}{n-1}\right)}_{\beta_2}f^2 \varphi+\left(1-\frac{n\beta_1}{n-1}\right)\nabla f\cdot\sigma\varphi\right\rangle dV.
    \end{split}
\end{equation}
Since $\beta>0$, we find that $\beta_1=\frac{1}{n(n-1)}-\frac{1}{\beta(n-1)^2}< \frac{1}{n(n-1)}$, and so
\begin{equation}
   \beta_2= 1-\frac{n\beta_1}{n-1}> 1-\frac{1}{(n-1)^2}>0.
\end{equation}
It follows that, provided $\pm f$ is sufficiently large on $\partial_\pm M^n$, the boundary terms of \eqref{n beta1} are nonpositive.

We now proceed by contradiction and assume that there exists an $\varepsilon>0$ such that
\begin{equation}
d(\partial_- M^n,\partial_+ M^n)> w:=\pi \sqrt{\frac{4\beta_2c(n-1)}{n\Lambda_c}}+\varepsilon=2\pi\sqrt{\frac{c}{\Lambda_c}\left(\frac{(4c-1)n+2-4c}{(4c-1)n+1}\right)}+\varepsilon.
\end{equation}
Next, define a sequence of bounded Lipschitz functions $f_j$ on $M^n$, which satisfy a certain differential inequality and have the property that $\pm f_j\to\infty$ on $\partial_\pm M^n$ as $j\to\infty$, in the following way. 
Let $r_\pm(x)=d(x,\partial_\pm M^n)$, and for each $j$ consider
\begin{equation}
    f_j(x)=\begin{cases}
    -\frac{\pi}{w}\cot\left( \frac{\pi}{w}r_-(x)+\frac{1}{j}\right) &\text{ if } r_-(x) \le \frac{w}{\pi}(\frac{\pi}{2}-\frac{1}{j})
   \\  \frac{\pi}{w}\cot\left( \frac{\pi}{w}r_+(x)+\frac{1}{j}\right) &\text{ if } r_+(x) \le \frac{w}{\pi} (\frac{\pi}{2}-\frac{1}{j})
   \\ 0 &\text{ otherwise}
   \end{cases}.
\end{equation}
For each $j$, we may apply Proposition \ref{p:spinorexistence} to obtain a nontrivial solution $\varphi_j$ to \eqref{e:spinorpde}. Now fix a compact subset $\Omega\subset\mathring{M}^n$, where $\mathring{M}^n$ denotes interior, such that for all sufficiently large $j$ we have
\begin{equation}\label{ofhoqihoihpoqh}
\begin{split}
   \frac{3}{2}f_j^2-|\nabla f_j|\ge 1 \quad\textnormal{\; on \;} M^n\setminus\Omega,
    \\ \beta_2 f_j^2-\beta_2|\nabla f_j|+\frac{n\Lambda_c}{4c(n-1)}\ge C_\varepsilon
    \quad\textnormal{\; on \;} M^n,
\end{split}
\end{equation}
where $C_\varepsilon >0$ depends on $\varepsilon$, $n$, $c$, and $\Lambda_c$.
Then equations \eqref{divergence} and \eqref{ofhoqihoihpoqh}, together with the boundary condition of \eqref{e:spinorpde} and sign of $f_j |_{\partial_{\pm} M^n}$, imply
\begin{equation}\label{e:fjomega}
    \int_{M^n\setminus\Omega} \left(\frac{1}{2}f_j^2+1\right) |\varphi_j|^2dV
     \le\int_{M^n\setminus\Omega}(2f^2_j-|\nabla f_j|)|\varphi_j|^2dV 
     \le  \int_{\Omega}(|\nabla f_j|-2f^2_j)|\varphi_j|^2dV .
\end{equation}
Note that $\max_{\Omega}|\varphi_j|\neq0$, otherwise this estimate implies that $\varphi_j$ vanishes globally.
Thus by appropriate rescaling, it may be assumed without loss of generality that $\max_{\Omega} |\varphi_j |=1$, and \eqref{ofhoqihoihpoqh} along with \eqref{e:fjomega} yield
\begin{equation}\label{aownofinohgq}
\int_{\Omega}|\varphi_j|^2+\int_{M^n \setminus\Omega} \left(\frac{1}{2}f_j^2+1\right) |\varphi_j|^2dV\leq  \left(\frac{n\Lambda_c}{4c(n-1)\beta_2}+1\right)|\Omega|.
\end{equation}
It then follows from \eqref{n beta1}, \eqref{ofhoqihoihpoqh}, and \eqref{aownofinohgq} that
\begin{align}\label{main estimate}
\begin{split}
&\frac{n}{4c(n-1)}\int_{M^n}|\nabla |\varphi_j||^2 dV
    + \int_{\partial M^n}\Upsilon_j |\varphi_j|^2dA
    \\ \le& \int_{M^n} \left(-\beta_2 f_j^2+\beta_2|\nabla f_j|\right)|\varphi_j|^2 dV +\frac{n}{n-1}\int_{M^n}\left(\frac{|R|}{4}+\gamma_n |R^{\mathcal{E}}|\right)|\varphi_j|^2dV
    \\ \leq & C_1
    \end{split}
\end{align}
for some constant $C_1$ independent of $j$, where  
\begin{equation}
\Upsilon_j=\min_{\partial M^n}\left(\beta_2 |f_j|-\frac{n}{2(n-1)}|H|\right)   
\end{equation}
which satisfies $\Upsilon_j \rightarrow \infty$ as $j\rightarrow \infty$.


The inequalities \eqref{aownofinohgq} and \eqref{main estimate} show that the sequence $|\varphi_j|$ is uniformly bounded in $H^1(M^n)$, and thus $|\varphi_j|$ weakly subconverges to a function $|\pmb{\varphi}|$ in $H^1(M^n)$ with strong convergence in $H^s(M^n)$ for any $s\in[\frac{1}{2},1)$, see \cite[Theorem 9.22]{Folland} or \cite[Corollary 7.2]{Hitch}. Moreover, since the trace map $\tau: H^s(M^n)\to H^{s-\frac{1}{2}}(\partial M^n)$ is continuous \cite[Proposition 3.8]{Hitch}, we find that $|\varphi_j|$ converges subsequentially to $\tau(|\pmb{\varphi}|)$ in $L^2(\partial M^n)$. However,
since $\Upsilon_j\rightarrow \infty$ we find that \eqref{main estimate} yields $\tau(|\pmb{\varphi}|)=0$ on $\partial M^n$, and hence $|\pmb{\varphi}|\in H^1_0(\mathring{M}^n)$. Then taking the limit in \eqref{n beta1} while utilizing weak lower semi-continuity of the $H^1$-norm, strong convergence in $L^2$, Fatou's lemma together with \eqref{ofhoqihoihpoqh}, and applying the definition of the $c$-spectral constant produces
\begin{align}
\begin{split}
0\ge& \int_{M^n}\left(\frac{n}{4c(n-1)}\left(|\nabla |\pmb{\varphi}||^2 +cR|\pmb{\varphi}|^2 \right)
+\left(\beta_2 f^2 -\beta_2 |\nabla f| -\frac{n\gamma_n}{n-1}|R^{\mathcal{E}}|\right)|\pmb{\varphi}|^2\right)dV\\
\geq &
\int_{M^n}\left(\beta_2 f^2-\beta_2|\nabla f|+\frac{n\Lambda_c}{4c(n-1)}-\frac{n\gamma_n}{n-1}|R^{\mathcal{E}}|\right)|\pmb{\varphi}|^2dV
        \\  \ge& \int_{M^n}\left(C_\varepsilon -\frac{\delta n\gamma_n}{n-1}\right)|\pmb{\varphi}|^2dV,
\end{split}
\end{align}
where in the last line we used Proposition \ref{p:ahatarea}. By choosing $\delta<< C_{\varepsilon}$, we arrive at a contradiction since $\max_{\Omega}|\pmb{\varphi}|=1$. It follows that
\begin{equation}
d(\partial_- M^n,\partial_+ M^n)\le 2\pi\sqrt{\frac{c}{\Lambda_c}\left(\frac{(4c-1)n+2-4c}{(4c-1)n+1}\right)}.
\end{equation}


\section{The $\mu$-Bubble Approach}
\label{sec4}

In this section we will establish Theorem \ref{T:mu bubble}, and for convenience will use the notation $\Lambda=\Lambda_{\frac{1}{2}}$. The result is trivial if $n=1$, and thus it will be assumed that $n>1$ below. Suppose that the conclusion of the theorem is false, then there exists $\varepsilon>0$ such that
\begin{equation}
d(\partial_- M^n,\partial_+ M^n)\ge\pi\sqrt{\frac{2n}{(n+1)\Lambda}}+2\varepsilon.
\end{equation}
Let $u$ be the positive principal eigenfunction associated to $\Lambda$, so that
\begin{equation}\label{ofoiqnoihq}
    \left(-\Delta+\frac{1}{2}R\right) u=\Lambda u\quad\text{ in $M^n$,}\quad\quad\quad\quad  u=0\quad\text{ on $\partial M^n$}.
\end{equation}
In order to obtain a band on which $u$ has a uniform positive lower bound, we may push in by a small amount from the boundary and consider
\begin{equation}
\check {M}^n=\left\{x\in M^n \mid d(x,\partial M^n)\geq\frac{\varepsilon}{2}\right\}.
\end{equation}
Note that it may be assumed without loss of generality that $\varepsilon$ is sufficiently small to guarantee that $\partial \check M^n$ is smooth and is divided into classes $\partial_{\pm}\check M^n$ corresponding with $\partial_{\pm} M^n$. 
Next denote $r_\pm(x)=d(x,\partial_{\pm}\check{M}^n)$, and for $0<\varepsilon_0<<\varepsilon$ define a potential function on the interior of $\check{M}^n$ by
\begin{align}
f_0(x)=
\begin{cases}
-\sqrt{\frac{2n\Lambda}{n+1}}\cot\left[\left(\sqrt{\frac{(n+1)\Lambda}{2n}}-\varepsilon_0\right)r_-(x)\right] & 0<r_-(x)\le\frac{\pi}{2} \left(\sqrt{\frac{(n+1)\Lambda}{2n}}-\varepsilon_0\right)^{-1}\\  
\sqrt{\frac{2n\Lambda}{n+1}}\cot\left[\left(\sqrt{\frac{(n+1)\Lambda}{2n}}-\varepsilon_0\right)r_+(x)\right]
& 0<r_+(x)\le \frac{\pi}{2} \left(\sqrt{\frac{(n+1)\Lambda}{2n}}-\varepsilon_0\right)^{-1}\\ 
0 &\text{elsewhere}
\end{cases}.
\end{align} 
Observe that $f_0$ is Lipschitz and limits to $\pm\infty$ on $\partial_{\pm}\check{M}^n$.
Let $B_{\sigma}(x)$ be the geodesic ball of radius $\sigma$ centered at an interior point $x\in\check{M}^n$, and set
\begin{equation}
    L_{f_0}(x)=\limsup_{\sigma\to 0} \text{Lip}_{B_\sigma(x)}(f_0).
\end{equation}
Then at all such points the following inequality holds
\begin{equation}\label{aoihfoihnah}
\frac{n+1}{2n}f_0^2-L_{f_0}+\Lambda\ge \sqrt{\frac{2n\Lambda}{n+1}}\varepsilon_0.
\end{equation}
This strictly positive lower bound for the left-hand side of \eqref{aoihfoihnah} is the impetus for introducing the constant $\varepsilon_0$.

We now seek to replace $f_0$ with a smooth approximation that agrees with it near $\partial\check{M}^n$. Let
$\check{M}^n_{r_0}=\{x\in \check{M}^n \mid r_\pm(x)\geq r_0\}$, where $r_0 >0$ is chosen sufficiently small so that $r_\pm$ are smooth (and hence $f_0$ is smooth) within $\check{M}^n\setminus\check{M}^n_{r_0}$. 
We may approximate $f_0$ by $f_{\delta}\in C^{\infty}(\check{M}^n_{r_0/2})$ which, for each small $\delta>0$, satisfies
\begin{equation}\label{oirhoqinohq}
  |f_0(x)-f_\delta(x)|\le \delta, \quad\quad \quad |\nabla f_\delta(x)|\le \mathrm{Lip}_{B_{\delta}(x)}(f_0)+\delta,
\end{equation}
with $x\in \check{M}^n_{r_0/2}$. Such an approximation $f_\delta$ may be constructed as in \cite[Theorem 2.2]{CRi}.
Furthermore given $\delta'>0$, the property \eqref{oirhoqinohq} implies that $L_{f_{\delta}}\leq L_{f_0} +\delta'$ for all $\delta$ sufficiently small, and therefore
we find that \eqref{aoihfoihnah} yields
\begin{equation}\label{aownfoiqnoihnq}
\frac{n+1}{2n}f_\delta^2-L_{f_\delta}+\Lambda\ge \sqrt{\frac{2n\Lambda}{n+1}}\varepsilon_0-2\delta'
\end{equation}
on $\check{M}^n_{r_0/2}$. Now let $\eta$ be a smooth nonnegative cut-off function on $\check{M}^n$ which is 1 on $\check{M}^n_{r_0}$ and zero on $\check{M}^n \setminus \check{M}^n_{r_0/2}$.
Define $f=\eta f_\delta+(1-\eta) f_0$ and observe that this function is smooth on $\check{M}^n$, agrees with $f_0$ on $\check{M}^n\setminus\check{M}^n_{r_0/2}$, and by virtue of \eqref{aoihfoihnah} and \eqref{aownfoiqnoihnq} it satisfies
\begin{equation}\label{80}
\frac{n+1}{2n}f^2-|\nabla f|+\Lambda\ge \sqrt{\frac{2n\Lambda}{n+1}}\varepsilon_0-3\delta'>0
\end{equation}
on $\check{M}^n$, if $\delta$ is sufficiently small and $\delta'<<\varepsilon_0$.

The next step is to introduce the warped $\mu$-bubbles which serve as the central geometric tool in this proof. We will closely follow the exposition developed in \cite[Section 3]{ChodoshLi} and \cite[Proposition 2.1]{Zhu}. Fix a Caccioppoli set $\Omega_0$ having smooth boundary with $\partial_+ \check{M}^n \subset \Omega_0$, and such that $\partial\Omega_0 \setminus\partial_+\check{M}^n$ lies within the interior of $\check{M}^n$. For instance, one may take $\Omega_0$ to be an appropriate sublevel set of the distance function $r_+$. For any Caccioppoli set $\Omega\subset \check M^n$ with symmetric difference $\Omega\Delta\Omega_0$ compactly contained within the interior of $\check{M}^n$, define the functional
\begin{equation}\label{mu bubble functional}
    \mathcal A_{u,f}(\Omega)=\int_{\partial^\ast\Omega}ud\mathcal{H}^{n-1}-\int_{\check{M}^n}(\chi_\Omega-\chi_{\Omega_0})fud\mathcal{H}^n
\end{equation}
where $\partial^\ast \Omega$ denotes the reduced boundary, $\chi_{\Omega}$ is the characteristic function of $\Omega$, and $d\mathcal{H}^n$ is the $n$-dimensional Hausdorff measure. Using that $f$ blows-up at $\partial_{\pm}\check{M}^n$, it may be shown that a minimizer $\check{\Omega}$ of $\mathcal{A}_{u,f}$ exists within this class of sets, and since $n\le7$ its boundary $\partial\check{\Omega}$ is smooth.
Using the fact that $\check{\Omega}\Delta\Omega_0$ does not intersect $\partial\check M^n$, we find that $\partial_+\check M^n$ lies within $\check{\Omega}$, and hence the smooth hypersurface $\Sigma^{n-1}:=\partial\check{\Omega}\setminus\partial_+\check M^n$ must separate $\partial_-\check M^n$ from $\partial_+\check M^n$. This surface is referred to as a warped $\mu$-bubble, see Figure \ref{fig2}.

\begin{remark}
Instead of approximating $f_0$ by a smooth function $f$ and citing the existence theory for $\mu$-bubbles in the smooth setting as done above, one may directly construct $C^{2,\varsigma}$-regular $\mu$-bubbles with respect to the Lipschitz potential function $f_0$. This existence result is carried out in Appendix \ref{appA}, which may be of independent interest.
\end{remark}

\begin{figure}[H]
\includegraphics[totalheight=4cm]{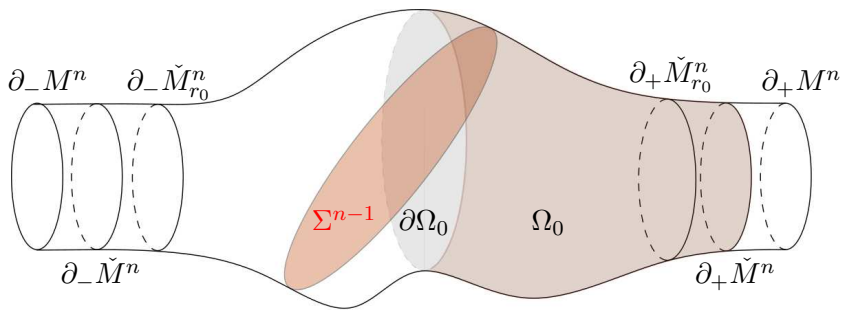}
\put(-305,83){$\partial_- M^n$}
\put(-260,83){$\partial_- \check M^n_{r_0}$}
\put(-285,10){$\partial_- \check M^n$}
\put(-20,83){$\partial_+ M^n$}
\put(-70,86){$\partial_+ \check M^n_{r_0}$}
\put(-45,10){$\partial_+ \check M^n$}
\put(-157,30){$\partial \Omega_0$}
\put(-107,30){$ \Omega_0$}
\put(-190,30){\textcolor{red}{$\Sigma^{n-1}$}}
\caption{
The relevant regions in the $\mu$-bubble approach.
}\label{fig2}
\end{figure}


A direct computation yields the first variation formula for the $\mu$-bubble
\begin{equation}\label{qohoqinohinqpoh3n}
    Hu-fu+\langle \nabla u,\nu\rangle=0,
\end{equation}
where $\nu$ is the unit outer normal to $\Sigma^{n-1}$, and $H$ is the mean curvature of $\Sigma^{n-1}$ with respect to $\nu$. Moreover, the second variation with test function $\phi\in C^{\infty}(\Sigma^{n-1})$ produces
\begin{align}
\begin{split}
    0\le &\int_{\Sigma^{n-1}}\left(-u\phi\Delta_\Sigma \phi-|A|^2\phi^2u-\Ric(\nu,\nu)\phi^2u+H\langle \nabla u,\nu\rangle \phi^2\right)dA\\
    &+\int_{\Sigma^{n-1}}\left(-f\phi^2\langle \nabla u,\nu\rangle-\phi^2u\langle \nabla f,\nu\rangle+\phi^2\nabla_{\nu\nu}u -\phi\langle \nabla_\Sigma  u,\nabla_\Sigma \phi   \rangle         \right)dA,
    \end{split}
\end{align}
where $A$ denotes the second fundamental form.
Utilizing the Gauss equations, the basic inequality $|A|^2\ge\frac1{n-1}H^2$, and the decomposition $\Delta u=\nabla_{\nu\nu}u+H\langle\nabla u,\nu\rangle+\Delta_\Sigma u$ gives rise to
\begin{align}
\begin{split}
    0\le &\int_{\Sigma^{n-1}}\left(-u\phi\Delta_\Sigma \phi-\frac{n}{2(n-1)}H^2\phi^2u-\frac12R\phi^2u+\frac12R_\Sigma\phi^2u\right)dA\\
    &+\int_{\Sigma^{n-1}}\left(-f\phi^2\langle \nabla u,\nu\rangle-\phi^2u\langle \nabla f,\nu\rangle+\phi^2(\Delta u-\Delta_\Sigma u) -\phi\langle \nabla_\Sigma u,\nabla_\Sigma \phi   \rangle          \right)dA,
    \end{split}
\end{align}
with $R_{\Sigma}$ denoting the scalar curvature of $\Sigma^{n-1}$. Equations \eqref{ofoiqnoihq} and \eqref{qohoqinohinqpoh3n}, along with the Cauchy-Schwarz inequality $|\langle\nabla f,\nu\rangle|\leq|\nabla f|$, then imply
\begin{align}\label{e:stab2}
\begin{split}
    0\le &\int_{\Sigma^{n-1}}\left(-u\phi\Delta_\Sigma \phi-\frac{n}{2(n-1)}(f-\langle \nabla \log u,\nu\rangle)^2\phi^2u-\Lambda\phi^2u+\frac12R_\Sigma\phi^2u\right)dA\\
    &+\int_{\Sigma^{n-1}}\left(-f\phi^2\langle \nabla u,\nu\rangle-\phi^2u\langle \nabla f,\nu\rangle-\phi^2\Delta_\Sigma u -\phi\langle \nabla_\Sigma u,\nabla_\Sigma \phi   \rangle          \right)dA\\
    \le &\int_{\Sigma^{n-1}}\left(-u\phi\Delta_\Sigma \phi-\frac{n}{2(n-1)}f^2\phi^2u-\frac{n}{2(n-1)}\langle \nabla \log u,\nu\rangle^2u\phi^2-\Lambda\phi^2u+\frac12R_\Sigma\phi^2u\right)dA\\
    &+\int_{\Sigma^{n-1}}\left(\frac1{n-1}u\phi^2f\langle \nabla \log u,\nu\rangle   +\phi^2u|\nabla f|-\phi^2\Delta_\Sigma u -\phi\langle \nabla_\Sigma u,\nabla_\Sigma \phi   \rangle          \right)dA.
    \end{split}
\end{align}
Observe that by Young's inequality
\begin{equation}
\frac{1}{n-1}u\phi^2f\langle\nabla \log u,\nu\rangle\leq  \frac{n}{2(n-1)}u\phi^2\langle\nabla\log u,\nu\rangle^2+\frac{1}{2n(n-1)} u\phi^2f^2,
\end{equation}
the therefore \eqref{e:stab2} becomes
\begin{align}\label{e:stab3}
\begin{split}
    0\le &\int_{\Sigma^{n-1}}\left(-u\phi\Delta_\Sigma \phi-\frac{n+1}{2n}f^2\phi^2u-\Lambda\phi^2u+\frac12R_\Sigma\phi^2u\right)dA\\
    &+\int_{\Sigma^{n-1}}\left(\phi^2u|\nabla f|-\phi^2\Delta_\Sigma u -\phi\langle \nabla_\Sigma u,\nabla_\Sigma \phi   \rangle          \right)dA.
    \end{split}
\end{align}
Next, let $\psi\in C^{\infty}(\Sigma^{n-1})$ and set $\phi=\psi u^{-\frac12}$ to find
\begin{align}
\begin{split}
    0   \le &\int_{\Sigma^{n-1}}\left(-u^{\frac12}\psi\Delta_\Sigma (\psi u^{-\frac12})-\psi^2u^{-1}\Delta_\Sigma u -u^{-\frac12}\psi\langle \nabla_\Sigma u,\nabla_\Sigma (\psi u^{-\frac12})  \rangle  \right)dA\\
    &+\int_{\Sigma^{n-1}}\left(|\nabla f| -\frac{n+1}{2n}f^2-\Lambda+\frac12R^\Sigma       \right)\psi^2dA.
    \end{split}
\end{align}
Finally if $n\geq 3$, integrating by parts, using Young's inequality again, and applying \eqref{80} produces
\begin{align}\label{pqjfoiqahnoihq}
\begin{split}
    0   \le &\int_{\Sigma^{n-1}}\left(|\nabla_\Sigma \psi|^2-\frac34\psi^2|\nabla_\Sigma \log u|^2+\psi \langle \nabla_\Sigma\psi,\nabla_\Sigma \log u\rangle \right)dA\\
    &+\int_{\Sigma^{n-1}}\left(|\nabla f| -\frac{n+1}{2n}f^2-\Lambda+\frac12R_\Sigma       \right)\psi^2dA\\
    \le&\int_{\Sigma^{n-1}}\left(\frac43|\nabla_\Sigma \psi|^2+\frac12R_\Sigma \psi^2\right)dA
    +\int_{\Sigma^{n-1}}\left(|\nabla f| -\frac{n+1}{2n}f^2-\Lambda      \right)\psi^2dA\\
    \le&\frac{2(n-1)}{n-2}\int_{\Sigma^{n-1}}\left(|\nabla_\Sigma \psi|^2+\frac{n-2}{4(n-1)}R_\Sigma \psi^2\right)dA
    +\int_{\Sigma^{n-1}}\left(|\nabla f| -\frac{n+1}{2n}f^2-\Lambda      \right)\psi^2dA\\
   <&\frac{2(n-1)}{n-2}\int_{\Sigma^{n-1}}\left(|\nabla_\Sigma \psi|^2+\frac{n-2}{4(n-1)}R_\Sigma \psi^2\right)dA,
    \end{split}
\end{align}
if $\psi$ is not identically zero.
Since $\psi$ was arbitrary, it follows that the principle eigenvalue of the conformal Laplacian of $(\Sigma^{n-1},g)$ is positive. In particular, $\Sigma^{n-1}$ admits a metric of positive scalar curvature. When $n=2$, the third line of \eqref{pqjfoiqahnoihq} is still valid, so that choosing $\psi=1$ yields the same conclusion. On the other hand, 
since $\Sigma^{n-1}$ separates $\partial_- \check{M}^n$ and $\partial_+ \check{M}^n$ and hence also $\partial_- M^n$ and $\partial_+ M^n$, the fact that $(M^n,\partial_\pm M^n)$ is overtorical implies that $\Sigma^{n-1}$ admits a nonzero degree map to $T^{n-1}$, see \cite[Lemma 6.2]{Rade1}. Since $n\leq7$, classical work of Schoen-Yau \cite{SY79} shows that $\Sigma^{n-1}$ cannot support positive scalar curvature metrics. From this contradiction we conclude that the desired inequality \eqref{aohfoiaopijhpoq} is valid.

\begin{remark}
Note that this last argument shows that overtorical bands are nonPSC-bands for $n\leq 8$. Indeed, \cite{SY79} continues to apply for this slightly extended range of dimensions.
\end{remark}


\section{The Spectral Cube Inequality}
\label{sec5}

In this section we will establish Theorem \ref{spectral cube}.
Let $u$ be the positive principal Dirichlet eigenfunction for the Riemannian cube, so that
\begin{equation}
\left(-\Delta +\frac12R\right)u=\Lambda_{\frac{1}{2}} u \quad\quad\text{ in }[-1,1]^n,\quad\quad\quad\quad u=0 \quad\quad\text{ on }\partial[-1,1]^n.
\end{equation}
Let $l,\varepsilon>0$ be parameters, and consider the higher dimensional cube $\tilde{M}^{n+1}_{l,\varepsilon}=[-l,l]\times[-1+\varepsilon,1-\varepsilon]^n$ with warped product metric $\tilde{g}=u^2dt^2+g$. Note that $u$ does not vanish on $\tilde{M}^{n+1}_{l,\varepsilon}$. Furthermore, observe that the scalar curvature \cite[(13a) page 214]{Neill} of $\tilde{g}$ satisfies
\begin{equation}
    \tilde{R}=-2u^{-1}\left(\Delta u-\frac12 Ru\right)=2\Lambda_{\frac{1}{2}} >0.
\end{equation}
By applying the pointwise version of Gromov's cube inequality \cite[Theorem 1.1]{WXY}, \cite[Section 3.8]{Gromov2} we then have
\begin{equation}
\sum_{i=0}^{n}\frac1{\ell_{i,l,\varepsilon}^2}\geq \frac{\Lambda_{\frac{1}{2}} (n+1)}{2\pi^2n},
\end{equation}
where $\ell_{i,l,\varepsilon}$ is the distance within $\tilde{M}^{n+1}_{l,\varepsilon}$ between the $i$th opposing faces of the cube, with $i=0$ corresponding to the $t$-direction. Moreover, since $\ell_{i,l,\varepsilon}$ is independent of $l$ for $i=1,\ldots,n$, and $\ell_{0,l,\varepsilon}\rightarrow\infty$ as $l\rightarrow\infty$, it follows that by passing to the limit
\begin{equation}
\sum_{i=1}^{n}\frac1{\ell_{i,\varepsilon}^2}\geq \frac{\Lambda_{\frac{1}{2}} (n+1)}{2\pi^2n},
\end{equation}
where $\ell_{i,\varepsilon}$ is the distance within $([-1+\varepsilon,1-\varepsilon]^n ,g)$ between the $i$th opposing faces of the cube. Finally, since $\ell_{i,\varepsilon}\rightarrow \ell_{i}$ as $\varepsilon\rightarrow 0$, the desired inequality is achieved.

\begin{remark}
In a similar fashion, this method also allows one to derive the spectral toric band inequality directly from the pointwise toric band inequality. Note however, that this warped product approach cannot deal with $c$-spectral constants for $c\neq \frac{1}{2}$, and it does not address the case of equality.
\end{remark}

\section{Black Hole Existence}
\label{sec6}

In this section Theorem \ref{bhexistence} will be established, comparison with the Schoen-Yau black hole existence result \cite{SY} will be discussed, and examples will be presented. The main steps in the proof of the existence of apparent horizons will follow the prescription of \cite{SY}, and thus here only an outline will be given with remarks provided to accommodate the higher dimensions and different radii. 

\subsection{Proof of Theorem \ref{bhexistence}}
Consider an initial data set $(M^n,g,k)$ as in the statement of the theorem, and assume by way of contradiction that it does not contain any closed properly embedded smooth apparent horizons.
Then there exists a regular solution to the Dirichlet problem for the Jang equation
\begin{equation}\label{qoiroiqnoihn}
\left(g^{ij}-\frac{f^{i}f^{j}}{1+|\nabla f|^{2}}\right)\left(\frac{\nabla_{ij}f}
{\sqrt{1+|\nabla f|^{2}}}-k_{ij}\right)=0\quad\text{ on }M^n,\quad\quad\quad f=0\quad\text{ on }\partial M^n,
\end{equation}
where $f^{i}=g^{ij}\partial_{j}f$ and $\nabla_{ij}f$ denotes the covariant Hessian. The existence is obtained from a limit
of solutions to the capillarity regularized equation, utilized by Schoen-Yau in the proof of the positive mass theorem \cite{SY2} in dimension 3. This was extended to dimensions $n\leq 7$ by Eichmair \cite[Proposition 7]{Eichmair2} in the asymptotically flat setting, using the theory of $C$-almost minimizing boundaries \cite[Appendix A]{Eichmair1}. The necessary tool needed to apply Eichmair's strategy to the Dirichlet problem \eqref{qoiroiqnoihn} is a 2-sided barrier construction at the boundary $\partial M^n$. This is explained for dimension 3 in \cite[page 11]{Yau}, and the same construction holds essentially without change in higher dimensions as long as the boundary is untrapped. Because the solution of Jang's equation represents a MOTS in $n+1$ dimensions, one might expect its singular set to be at best codimension 7. However, better regularity properties prevail as it is a graph \cite[Remark 4.1, pages 568-569]{Eichmair1}, leaving its singular set to be at least codimension 8.

Consider now the Jang metric $\bar{g}=g+df^2$ on $M^n$. Its scalar curvature \cite[(10)]{Eichmair2} satisfies the identity
\begin{equation}\label{oqrt0oihnqoibhoqhb}
\bar{R}=2(\mu-J(v))+
|A-k|_{\bar{g}}^{2}+2|X|_{\bar{g}}^{2}
-2\mathrm{div}_{\bar{g}}(X),
\end{equation}
where $A$ is the second fundamental form of the graph $t=f(x)$ in the product manifold $(M^n \times\mathbb{R},g+dt^2)$, $\mathrm{div}_{\bar{g}}$
is the divergence operator with respect to $\bar{g}$, and $v$ and
$X$ are 1-forms given by
\begin{equation}
v_{i}=\frac{f_{i}}{\sqrt{1+|\nabla f|^{2}}},\quad\text{
}\text{ }\text{ }\text{ }
X_{i}=\frac{f^{j}}{\sqrt{1+|\nabla f|^{2}}}(A_{ij}-k_{ij}).
\end{equation}
Let $u>0$ be the principal Dirichlet eigenfunction of $-\Delta_{\bar{g}}+\frac{1}{2}\bar{R}$ on $\Omega$. Then
multiplying \eqref{oqrt0oihnqoibhoqhb} by $u^2$ and integrating by parts produces
\begin{equation}\label{qohfoihnqoiqh}
\int_{\Omega}\left((\mu-|J|)+\frac{1}{2}|A-k|_{\bar{g}}^2+|X+\nabla \log u|_{\bar{g}}^2\right)u^2 dV_{\bar{g}}
\leq \int_{\Omega}\left(|\nabla u|^2_{\bar{g}}+\frac{1}{2}\bar{R}u^2\right)dV_{\bar{g}}.
\end{equation}
Notice that it is not possible for both $|A-k|_{\bar{g}}$ and $|X+\nabla \log u|_{\bar{g}}$ to vanish on $\Omega$,
otherwise this would imply that $u$ is constant. Therefore the integral involving these two terms gives a strictly
positive contribution to the left-hand side. Using that $\mu-|J|\geq\Lambda$ on $\Omega$, we conclude that the corresponding
principal eigenvalue satisfies $\bar{\Lambda}\geq (1+\varepsilon)\Lambda$ for some $\varepsilon>0$ sufficiently small.

If $N^n \hookrightarrow (\Omega,\bar{g})$ is an isometrically immersed nonPSC-band or cube, then by utilizing the pullback of $u$ on $N^n$ in the proofs of 
Theorems \ref{T:mu bubble} and \ref{spectral cube}, we find that the band or cubical-width satisfies
\begin{equation}
\text{$\bar{g}$-width}\leq \pi\sqrt{\frac{2n}{(n+1)\bar{\Lambda}}}\leq \pi\sqrt{\frac{2n}{(n+1)(\varepsilon+1)\Lambda}}.
\end{equation}
It then follows from the definition of torical and cubical-radius, that
\begin{equation}\label{eqh0h09h09yqh}
\overline{\mathrm{Rad}}(\Omega)\leq \pi\sqrt{\frac{2n}{(n+1)(\varepsilon+1)\Lambda}}<\pi\sqrt{\frac{2n}{(n+1)\Lambda}},
\end{equation}
where $\overline{\mathrm{Rad}}$ denotes the radius with respect to the Jang metric. Furthermore, since $\bar{g}$ is larger than $g$ we have $\mathrm{Rad}(\Omega)\leq\overline{\mathrm{Rad}}(\Omega)$. However, this combined with \eqref{eqh0h09h09yqh} leads to a contradiction with the assumption \eqref{q3j0qhp0hjpoihjnaoi}. We conclude that $M^n$ must contain a closed properly embedded smooth apparent horizon $\mathcal{S}^{n-1}$.

Lastly, to verify the last claims of Theorem \ref{bhexistence}, note that the apparent horizon may be identified via blow-up of the Jang equation. Moreover, the same manipulations that give rise to \eqref{qohfoihnqoiqh}, provide an analogous stability type inequality on the Jang surface where $u$ is replaced by smooth functions with compact support. With standard arguments, as in \cite[Proposition 9]{Eichmair2}, this stability property is inherited by the apparent horizon. Thus, if $\mu-|J|\geq\lambda>0$ on $\mathcal{S}^{n-1}$, then the principal eigenvalue of $-\Delta_{\mathcal{S}^{n-1}}+\frac{1}{2}R_{\mathcal{S}^{n-1}}$ is not less than $\lambda$. By Theorems \ref{T:mu bubble} and \ref{spectral cube}, it follows that 
\begin{equation}
\mathrm{Rad}(\mathcal{S}^{n-1})\leq \pi\sqrt{\frac{2(n-1)}{n\lambda}}.
\end{equation}
Moreover, if $c_n =\frac{n-2}{4(n-1)}$ is the dimensional constant from the conformal Laplacian, then since $2\leq c_n^{-1}$ the same arguments show that the principal eigenvalue of $-\Delta_{\mathcal{S}^{n-1}}+c_n R_{\mathcal{S}^{n-1}}$ is positive. Hence $\mathcal{S}^{n-1}$ is of positive Yamabe type.

\subsection{Comparison to Schoen-Yau result}

As discussed in the introduction, the 3-dimensional Schoen-Yau black hole existence result in \cite{SY} relies on a different notion of radius than those used in this article. In order to compare Theorem \ref{bhexistence} with their result, we will in this subsection compare the torical-radius $\mathrm{Rad}_t$ and the Schoen-Yau radius $\mathrm{Rad}_{sy}$. A preliminary observation reveals that the neighborhoods used to build the Schoen-Yau radius are related to nonPSC-bands.

\begin{proposition}\label{aoifnoiqnophiqpojh}
Let $(\Omega^3,g)$ be a compact Riemannian 3-manifold with (possibly empty) boundary $\partial\Omega^3$, and assume that $\Gamma\subset\mathring{\Omega}^3$ is a smooth simple closed curve which bounds a disc $D\subset\Omega^3$. If $\Gamma$ does not bound a disc within the distance neighborhood $N_r=\{x\in\Omega^3 \mid d(x,\Gamma)< r\}$ and $N_r\cap\partial\Omega^3=\emptyset$, then any embedded hypersurface in $N_r$ separating $\Gamma$ from $\partial N_r$ must have a component of nonzero genus. 
\end{proposition}

\begin{proof}
Suppose that $N_r\subset \Omega^3$ is a distance neighborhood of the curve $\Gamma$, such that there is no disc in $N_r$ bounded by $\Gamma$ and $N_r\cap\partial\Omega^3 =\emptyset$. Note that $\partial N_r\neq\emptyset$, otherwise $N_r=\Omega^3$ and $D$ would lie within $N_r$. Proceeding by contradiction, let us suppose that there is an embedded hypersurface $\Sigma^2 \hookrightarrow N_r$ which separates $\Gamma$ from $\partial N_r$, and has the property that each of its components is a 
2-sphere. Let $N'_r$ denote the component of $N_r\setminus \Sigma^2$ that contains $\Gamma$, and note that its boundary consists of spheres. We may assume without loss of generality that $D$ intersects $\Sigma^2$ transversely, and will denote by $D'$ the component of $D\cap N'_r$ which contains $\Gamma$. Then $\partial D' \setminus \Gamma$ consists of a finite number of circles within $\Sigma^2$. Since each component of $\Sigma^2$ is a sphere, these circles bound discs within $\Sigma^2$ which may be used to cap off $D'$. Thus, the union of $D'$ with these caps produces a disc that lies within $N_r$ and is bounded by $\Gamma$, yielding a contradiction.  
We conclude that $\Sigma^2$ must contain at least one component of nonzero genus.
\end{proof}

The main observation gives the desired relation between the two notions of radii. 
In particular, this comparison implies that Theorem \ref{bhexistence} recovers \cite[Theorem 2]{SY} in dimension 3. 

\begin{lemma}
Let $(\Omega^3,g)$ be a compact Riemannian 3-manifold. Then $\mathrm{Rad}_t(\Omega^3)\geq \mathrm{Rad}_{sy}(\Omega^3)$.
\end{lemma}

\begin{proof}
Let $\Gamma\subset \mathring{\Omega}^3$ be a smooth simple closed curve. According to a version of Sard's theorem \cite{Rifford}, the set of critical values for the distance function from $\Gamma$ is of measure zero, and thus when computing the Schoen-Yau radius it suffices to restrict attention to regular values. Consider such a regular value $r$, then $\partial N_r$ is a Lispchitz hypersurface \cite{Rifford}, and may therefore be approximated with a smooth hypersurface that is homologous and arbitrarily close to $\partial N_r$ by, for instance, running mean curvature flow for a short time \cite{EckerHuisken}. In particular, we may assume without loss of generality that $N_r$ possesses a smooth boundary. Let $\varepsilon>0$ and consider the annular distance neighborhood $N_r \setminus N_{\varepsilon}$. By Proposition \ref{aoifnoiqnophiqpojh}, for all sufficiently small $\varepsilon$ this annular distance neighborhood defines a nonPSC-band of width $r-\varepsilon$. If $\mathbf{r}$ is the supremum of
values $r$ with the property that the $r$-distance neighborhood from $\Gamma$ does not intersect $\partial\Omega^3$, and $\Gamma$ does not bound a disc in this neighborhood, then since $\varepsilon$ may be taken arbitrarily small we have
$\mathbf{r}\leq\mathrm{Rad}_t(\Omega^3)$. Furthermore, since $\mathrm{Rad}_{sy}(\Omega^3)$ is the supremum of $\mathbf{r}$ among all $\Gamma$, it follows that $\mathrm{Rad}_{sy}(\Omega^3)\leq\mathrm{Rad}_t(\Omega^3)$.
\end{proof}

\subsection{Examples}
As explained in the introduction, for the class of maximal initial data sets Theorem \ref{bhexistence} is vacuous. However, here we show by explicit construction that it is straightforward to find examples that satisfy the hypotheses of this result, and in fact that they are ubiquitous. Let $(M^n,g)$ be an arbitrary complete asymptotically flat Riemannian manifold, and consider an embedded cube $[-1,1]^n\hookrightarrow M^n$. Now define a symmetric 2-tensor $k=Fg$, where $F$ is a smooth compactly supported function on $M^n$, with $F\equiv C>>1$ inside the cube.
Then $(M^n,g,k)$ is an asymptotically flat initial data set whose energy and momentum densities are given by $\mu=\frac12(R+(n^2-n)C^2)$ and $J=0$, inside the cube. Therefore if $3\leq n\leq 7$, then by choosing the constant $C$ to be sufficiently large we find that the assumptions of Theorem \ref{bhexistence} and Corollary \ref{cor main} are satisfied, which yields a closed properly embedded smooth apparent horizon within $(M^n,g,k)$. The above construction can also be adapted for the torical-radius version of the theorem.

\appendix

\section{Existence and Regularity of Warped $\mu$-Bubbles}
\label{appA}

In this section we discuss the existence and regularity of warped $\mu$-bubbles with Lipschitz potential function $f$, which does not appear to be in the literature. Previous results on this topic have assumed a smooth potential function, however the most natural choices for $f$ in applications are often merely Lipschitz since they involve distance functions. The notation here will be consistent with that of Section \ref{sec4}, with $\check{M}^n$ replaced with $M^n$.

\begin{proposition}
     Let $(M^n,\partial_\pm M^n,g)$ be an $n$-dimensional Riemannian band with $n\leq 7$. Suppose that $u\in C^\infty(M^n)$ is strictly positive, and $f\in\mathrm{Lip}_{loc}(M^n)$ satisfies $f\to\pm\infty$ on $\partial_\pm M^n$. 
     Then for any $\varsigma\in (0,1)$ there exists a $C^{2,\varsigma}$ warped $\mu$-bubble $\Sigma^{n-1}=\partial\Omega \setminus\partial_{+}M^n$, where $\Omega$ minimizes the functional $\mathcal{A}_{u,f}$ of \eqref{mu bubble functional} among Caccioppoli sets whose symmetric difference with $\Omega_0$ is compactly contained within the interior of $M^n$.
\end{proposition}

\begin{proof}
The existence theory for $\mu$-bubbles relies on the compactness theorem for Caccioppoli sets, and extends without any adjustment to the non-smooth setting. More precisely, it follows from \cite[Section 3]{ChodoshLi} that a minimizing Caccioppoli set $\Omega$ exists, whose reduced boundary $\partial^\ast\Omega\setminus\partial_+ M^n =\Sigma^{n-1}$ does not intersect $\partial M^n$.
Moreover, it is straightforward to show that $\Sigma^{n-1}$ satisfies the \emph{$C$-almost minimizing} property, and therefore according to \cite[Theorem A.1]{Eichmair1} this surface is $C^{1,\varsigma}$ smooth. Alternatively, 
as in \cite[Theorem 2.2]{ZhouZhu}, we may follow the arguments contained in \cite[Section 3]{Morgan} to obtain the same conclusion. Writing $\Sigma^{n-1}$ locally as graph, we find that the graph function weakly satisfies the second order elliptic equation
\begin{equation}
H=f-\langle\nabla\log u, \nu\rangle.
\end{equation}
Since the potential function $f$ is Lipschitz, the normal $\nu$ is $C^{0,\varsigma}$, and the weight function $u>0$ is smooth, standard Schauder theory yields $C^{2,\varsigma}$ regularity for the $\mu$-bubble.
\end{proof}





\end{document}